\documentclass[12pt, a4paper]{amsart}

\usepackage[top=25truemm,bottom=25truemm,left=25truemm,right=25truemm]{geometry}

\usepackage{amssymb, amsfonts, mathrsfs, amsthm, color, bm, bbm, url}
\usepackage[all]{xy}

\usepackage{amsmath}

\theoremstyle{plain}	
\newtheorem{theorem}{Theorem}[section]
\newtheorem{proposition}[theorem]{Proposition}
\newtheorem{lemma}[theorem]{Lemma}

\newtheorem{corollary}[theorem]{Corollary}

\theoremstyle{definition}
\newtheorem{definition}[theorem]{Definition}
\newtheorem{remark}[theorem]{Remark}
\newtheorem{example}[theorem]{Example}

\newtheorem{condition}[theorem]{Condition}

\newcommand{\DS}{\displaystyle}

\newcommand{\SC}{\scriptstyle}
\newcommand{\SSC}{\scriptscriptstyle}

\DeclareMathOperator{\Aut}{Aut}

\DeclareMathOperator{\End}{End}

\DeclareMathOperator{\Frob}{Frob}
\DeclareMathOperator{\Gal}{Gal}

\DeclareMathOperator{\Hom}{Hom}

\DeclareMathOperator{\id}{id}
\renewcommand{\Im}{\mathop{\rm Im}}
\DeclareMathOperator{\iHom}{\mathcal{H}om}
\DeclareMathOperator{\Ker}{Ker}

\DeclareMathOperator{\Lie}{Lie}

\DeclareMathOperator{\QHom}{QHom}
\DeclareMathOperator{\QEnd}{QEnd}
\DeclareMathOperator{\ord}{ord}

\DeclareMathOperator{\rk}{rk}
\DeclareMathOperator{\Spec}{Spec}
\DeclareMathOperator{\Spf}{Spf}

\newcommand{\ra}{\rightarrow}

\newcommand{\hra}{\hookrightarrow}
\newcommand{\thra}{\twoheadrightarrow}
\newcommand{\lra}{\longrightarrow}

\newcommand{\lrai}{\overset{\sim}{\lra}}

\newcommand{\dbl}{{\mathchoice{\mbox{\rm [\hspace{-0.15em}[}}
                              {\mbox{\rm [\hspace{-0.15em}[}}
                              {\mbox{\scriptsize\rm [\hspace{-0.15em}[}}
                              {\mbox{\tiny\rm [\hspace{-0.15em}[}}}}
\newcommand{\dbr}{{\mathchoice{\mbox{\rm ]\hspace{-0.15em}]}}
                              {\mbox{\rm ]\hspace{-0.15em}]}}
                              {\mbox{\scriptsize\rm ]\hspace{-0.15em}]}}
                              {\mbox{\tiny\rm ]\hspace{-0.15em}]}}}}
\newcommand{\dpl}{{\mathchoice{\mbox{\rm (\hspace{-0.15em}(}}
                              {\mbox{\rm (\hspace{-0.15em}(}}
                              {\mbox{\scriptsize\rm (\hspace{-0.15em}(}}
                              {\mbox{\tiny\rm (\hspace{-0.15em}(}}}}
\newcommand{\dpr}{{\mathchoice{\mbox{\rm )\hspace{-0.15em})}}
                              {\mbox{\rm )\hspace{-0.15em})}}
                              {\mbox{\scriptsize\rm )\hspace{-0.15em})}}
                              {\mbox{\tiny\rm )\hspace{-0.15em})}}}}

\newcommand{\alg}{\mathrm{alg}}
\newcommand{\ab}{\mathrm{ab}}

\newcommand{\crys}{{\rm crys}}

\newcommand{\sep}{{\rm sep}}
\renewcommand{\ss}{{\rm ss}}

\newcommand{\tor}{{\rm tor}}
\newcommand{\ur}{{\rm ur}}

\newcommand{\zb}{\dbl z \dbr}     
\newcommand{\zp}{\dpl z \dpr}     
\newcommand{\zf}{\frac{1}{z}}       

\newcommand{\zz}{z-\zeta}          
\newcommand{\zzb}{\dbl z-\zeta \dbr}  
\newcommand{\zzp}{\dpl z-\zeta \dpr}
\newcommand{\zzf}{\frac{1}{z-\zeta}}

\newcommand{\sig}{\sigma}
\renewcommand{\phi}{\varphi}
\newcommand{\lam}{\lambda}

\newcommand{\hsig}{\wh{\sig}}

\let\setminus\smallsetminus

\renewcommand{\ast}{^{\SC*}}
\newcommand{\wc}{\check}
\newcommand{\es}{\enspace}

\newcommand{\inv}{^{\SSC-1}}

\newcommand{\col}{\colon}
\newcommand{\ul}[1]{{\underline{#1\hspace{-0.18em}}\hspace{0.18em}}}
\newcommand{\ol}[1]{{\overline{#1}}}
\newcommand{\wh}[1]{{\hat{#1}}}
\newcommand{\wt}[1]{{\tilde{#1}}}

\newcommand{\ot}{\otimes}
\newcommand{\op}{\oplus}

\usepackage{ifthen}

\newcommand{\ilim}[1][]{\ifthenelse{\equal{#1}{}}
{\DS \lim_{\longleftarrow}}
{\DS \lim_{\underset{#1}{\longleftarrow}}}
}

\newcommand{\dlim}[1][]{\ifthenelse{\equal{#1}{}}
{\DS \lim_{\longrightarrow}}
{\DS \lim_{\underset{#1}{\longrightarrow}}}
}

\makeatletter 
\@tempcnta\z@
\loop\ifnum\@tempcnta<26
\advance\@tempcnta\@ne
\expandafter\edef\csname rm\@Alph\@tempcnta\endcsname{\noexpand\mathrm{\@Alph\@tempcnta}}
\expandafter\edef\csname s\@Alph\@tempcnta\endcsname{\noexpand\mathscr{\@Alph\@tempcnta}}
\expandafter\edef\csname b\@Alph\@tempcnta\endcsname{\noexpand\mathbb{\@Alph\@tempcnta}}
\expandafter\edef\csname c\@Alph\@tempcnta\endcsname{\noexpand\mathcal{\@Alph\@tempcnta}}
\expandafter\edef\csname rm\@alph\@tempcnta\endcsname{\noexpand\mathrm{\@alph\@tempcnta}}
\expandafter\edef\csname fr\@alph\@tempcnta\endcsname{\noexpand\mathfrak{\@alph\@tempcnta}}
\expandafter\edef\csname fr\@Alph\@tempcnta\endcsname{\noexpand\mathfrak{\@Alph\@tempcnta}}
\expandafter\edef\csname bs\@Alph\@tempcnta\endcsname{\noexpand\boldsymbol{\@Alph\@tempcnta}}
\expandafter\edef\csname bs\@alph\@tempcnta\endcsname{\noexpand\boldsymbol{\@alph\@tempcnta}}
\expandafter\edef\csname bf\@Alph\@tempcnta\endcsname{\noexpand\mathbf{\@Alph\@tempcnta}}
\expandafter\edef\csname bf\@alph\@tempcnta\endcsname{\noexpand\mathbf{\@alph\@tempcnta}}
\repeat

 \makeatletter
           
    \@addtoreset{equation}{section}
  \makeatother

\begin{document}

\author{Yoshiaki Okumura}
\title{Torsion of rank-two $A$-motives values in odd characteristic cyclotomic towers}
\date{}

\maketitle

\begin{abstract}
For rank-two $A$-motives defined over local fields with odd characteristic, we give an analogue of a theorem of Imai stating that abelian varieties with good reduction over $p$-adic fields have only finitely many torsion points values in cyclotomic towers.
This implies the finiteness of torsion points of abelian Anderson $A$-modules.
For rank-two Drinfeld $A$-modules over global function fields, 
we also give an analogue of a theorem of Ribet on torsion points of abelian varieties values in maximal cyclotomic extensions of number fields.
\end{abstract}


\pagestyle{myheadings}
\markboth{Y.\ Okumura}{Imai's theorem for rank-two $A$-motives}

\renewcommand{\thefootnote}{\fnsymbol{footnote}}
\renewcommand{\thefootnote}{\arabic{footnote}}
\renewcommand{\thefootnote}{\fnsymbol{footnote}}
\renewcommand{\thefootnote}{\arabic{footnote}}

\setcounter{tocdepth}{1}
\tableofcontents  



\section{Introduction}

\subsection{Motivation} 
Let us begin with a review of Imai's theorem and its generalizations.
For an abelian variety $\cA$ defined over a finite extension $K$ of the $p$-adic number field $\bQ_p$ and an extension field $\cK/K$, 
the group $\cA(\cK)$ of $\cK$-valued points is one of the central considerations in arithmetic geometry. 
Mattuck proved in \cite{Mat55} that its torsion subgroup $\cA(\cK)_\tor$ is finite if $\cK/K$ is a finite extension.
Thus, for various extensions $\cK/K$ with infinite degree, it is natural to ask whether the torsion subgroup $\cA(\cK)_\tor$ is finite or not.
 Imai \cite{Ima75} proved that $\cA(\cK)_\tor$ is finite if $\cA$ has potentially good reduction and $\cK=K(\mu_{p^\infty})$, where $\mu_{p^\infty}$ is the set of all $p$-power roots of unity.
Kubo and Taguchi \cite{KT13} gave a finiteness theorem on \'etale cohomology groups of proper smooth varieties with potentially good reduction, which implies 
a generalization of Imai's theorem for the field $\cK=K(K^{1/p^\infty})$, where $K^{1/p^\infty}$ is the set of all $p$-power roots of all elements in $K$.
Other generalizations are given by Ozeki \ \cite{Oze20, Oze23}; he replaces $K(\mu_{p^\infty})$ by some extensions of $K$  coming from Lubin-Tate formal groups with respect to uniformizing parameters with mild conditions.  

Our aim is to establish a positive characteristic analogue of Imai's theorem.
In function field arithmetic, \textit{Drinfeld $A$-modules} and \textit{abelian Anderson $A$-modules} (Definition \ref{def.And}) play a role of elliptic curves and higher-dimensional abelian varieties, respectively; see \cite{Dri74, DH87, And86, Har19}.
It is also known that  an equal characteristic analogue of $\bQ_p(\mu_{p^\infty})$ is given by \textit{$z$-adic cyclotomic extensions} explained below (cf.\ Example \ref{example.cyclotomic} and \cite[\S 1.3]{Har09}).
It should be emphasized that, unlike abelian varieties,  
abelian Anderson $A$-modules 
can be embedded fully faithfully into a category of more general 
objects so called \textit{$A$-motives} (Definition \ref{def.Amot}), which
 is originally introduced by Anderson \cite{And86} as \textit{$t$-motives}.
 This allows us to handle the finiteness of torsion points of abelian Anderson $A$-modules by Galois representations attached to $A$-motives. 
 In this paper, under some assumptions, 
we give a finiteness theorem for Galois representations of $A$-motives and torsion points values in $z$-adic cyclotomic extensions.

\subsection{Notation}
We will use the following notation in this paper, except in \S 5.

Let $p$ be a prime number and $\bF_q$ a finite field with $q$ elements of characteristic $p$.
Let $Q$ be a global function field with constant field  $\bF_q$.
Fix a place $\infty$ of $Q$ and denote by $A$ the ring of elements of $Q$ regular outside $\infty$.
Then other places of $Q$ are identified with (non-zero) primes of $A$.
For each prime $\frl$ of $A$, 
let $\ord_\frl$ be the normalized valuation of $Q$ associated to $\frl$.
We write $\bF_\frl=A/\frl$ for the residue field at $\frl$ and put $\deg\frl:=[\bF_\frl:\bF_q]$. 
We also denote by 
$\DS A_\frl$
 the $\frl$-adic completion of $A$ and by $Q_\frl$ its field of fractions.

Let $L$ be a local field containing $\bF_q$.
Denote by $R$ its valuation ring, and by $k$ its residue field.
Let $\gamma \colon A \to R$ be a fixed injective $\bF_q$-algebra homomorphism.
We regard $L$ as an \textit{$A$-field} via $\gamma\colon A\to R\subset L$.
Define the \textit{$A$-field characteristic} of $L$ by the kernel
\[
\frp:=\Ker(A\overset{\gamma}{\to} R\thra k)
\]
  of the composite of $\gamma$ and the reduction map $R\thra k$.
 We set $\wh q:=\#\bF_\frp=q_{}^{\deg \frp}$.
We take a uniformizing parameter $z\in Q$ at $\frp$ (i.e., an element with $\ord_\frp(z)=1$).
 This allows us to identify $A_\frp=\bF_\frp\zb$ and $Q_\frp=\bF_\frp\zp$.
Set $\zeta:=\gamma(z)\in R$.
By continuity,  $\gamma$ can be extended to $\gamma\colon Q_\frp\to L$ whose image is 
\[
F_\frp:=\gamma(Q_\frp)=\bF_\frp\dpl \zeta \dpr \subset L.
\]
Fix an algebraic closure $L^\alg$ of $L$ and write $F_\frp^\sep\subset L^\sep$ for the separable closures of $F_\frp$ and $L$ in $L^\alg$, respectively.
Let $\cG_{L}=\Gal(L^\sep/L)$ be the absolute Galois group of $L$ and write 
$\cI_L$ for the inertia subgroup, that is the kernel of $\cG_L\thra \cG_k:=\Gal(k^\sep/k)$.

\subsection{$z$-adic cyclotomic theory}
For the element $\zeta=\gamma(z)\in R$, 
 let $\{\ell_n\}_{n=0}^\infty$ be a system of elements of $F_\frp^\sep$ such that 
$
\ell_0^{\wh{q}-1}=-\zeta\es\es \text{and}\es\es \ell_n^{\wh{q}}+\zeta \ell_n=\ell_{n-1}.
$
We call
\[
F_{\frp,\infty}:=F_\frp(\ell_n \colon n\in \bZ_{\geq 0})
\]
the \textit{$z$-adic cyclotomic extension} of $F_\frp$.
 It is 
a Lubin-Tate extension of $F_\frp$ associated to the uniformizing parameter $\zeta=\gamma(z)$ of $F_\frp$
and so is an analogue of $\bQ_p(\mu_{p^\infty})$.
We also put
\[
L_\infty:=LF_{\frp,\infty}=L(\ell_n \colon n\in \bZ_{\geq 0}).
\]
This is abelian over $L$ and so $\cG_{L_\infty}=\Gal(L^\sep/L_\infty)$ is a normal subgroup of $\cG_L$.

For the power series 
$
\ell^+:=\sum_{n=0}^\infty \ell_nz^n \in F_\frp^\sep \zb,
$
 there is an isomorphism of topological groups called the \textit{$z$-adic cyclotomic character}
\[
\chi_z\colon \Gal(F_{\frp,\infty}/F) \overset{\sim}{\lra} \bF_\frp\zb^\times=A_\frp^\times
\]
determined by $g\cdot\ell^+:=\sum_{n=0}^\infty g(\ell_n)z^n=\chi_z(g) \ell^+$ in $F_\frp^\sep\dbl z \dbr$ for all $g\in \Gal(F_{\frp,\infty}/F_\frp)$.
This plays the role of the $p$-adic cyclotomic character $\chi_p\colon \Gal(\bQ_p(\mu_{p^\infty})/\bQ_p)\overset{\sim}{\to} \bZ_p^\times$.
So the one-dimensional $Q_\frp$-representation $Q_\frp(n):=Q_\frp\cdot \mathbf{e}$ 
whose $\cG_{F_\frp}$-action is given by 
\[
g\cdot\mathbf{e}=\chi_z(g)^n\mathbf{e}\es\es \text{for}\es\es g\in \cG_{F_\frp}
\]
 is an analogue of 
the $n$-th Tate twist $\bQ_p(n)$.

It is known (Example \ref{example.cyclotomic}) that $Q_\frp(n)$ is a \textit{$z$-adic crystalline representation}
in \textit{Hodge-Pink theory}, an analogue of $p$-adic Hodge theory.
For our study, it is important to compute the \textit{weights} of $A$-motives and $z$-adic crystalline representations (Definitions \ref{definition.wmot} and \ref{definition.wd}).
To ensure the integrality of the weight of $Q_\frp(n)$ (see Lemma \ref{lemma.intweight}), 
we consider the following condition.

\begin{condition}\label{condition.main}
The integer $\sum_{\frl\neq\frp}\deg\frl\ord_\frl(z)$ is divisible by $\deg\frp$, 
 where $\frl$ runs through all primes of $A$ with $\frl\neq \frp$.
\end{condition}

In particular, Condition \ref{condition.main} holds if $\frp=(z)$, because 
$\sum_{\frl\neq\frp}\deg\frl\ord_\frl(z)=0$ in this case.


\subsection{Main theorem}
Let $\ul{M}$ be an $A$-motive over $L$.
Then for each prime $\frl$ of $A$, one can get a free $A_\frl$-module of finite rank and a finite-dimensional $Q_\frl$-vector space,
\[
\wc T_\frl \ul{M}\es\es\es \text{and}\es\es\es \wc V_\frl\ul{M}:=\wc T_\frl\ul{M}\otimes_{A_\frl}Q_\frl,
\]
called the \textit{$\frl$-adic realization} and the \textit{rational $\frl$-adic realization} of $\ul M$, respectively.
They have a continuous $\cG_L$-action.
We put
\[
\wc V\ul{M}:=\prod_{\frl}\wc V_\frl\ul{M} \es\es\supset\es\es \wc T\ul{M}:=\prod_{\frl}\wc T_\frl\ul{M}, 
\]
where $\frl$ runs through all primes of $A$. 
Here $\cG_L$ acts diagonally on both $\wc V\ul{M}$ and $\wc T\ul{M}$. 
Our central interest is 
whether the $\cG_{L_\infty}$-fixed submodule 
$(\wc V\ul{M}/\wc T\ul{M})^{\cG_{L_\infty}}$ of $\wc V\ul{M}/\wc T\ul{M}$ is finite or not.
The next is our main theorem.

\begin{theorem}[{$=$ Theorem \ref{theorem.main}}]\label{thm.intro}
Let $\ul{M}$ be an
$A$-motive over $L$ of rank two with good reduction and 
$z\in Q$ a uniformizing parameter at $\frp$ satisfying  Condition \ref{condition.main}.
If $p\neq 2$ and $\ul{M}$ has no integral weights (in the sense of Definition \ref{definition.wmot}), then $(\wc V\ul{M}/\wc T\ul{M})^{\cG_{L_\infty}}$ is finite.
\end{theorem}

The finiteness of  $(\wc V_\frl\ul{M}/\wc T_\frl\ul{M})^{\cG_{L_\infty}}$ is equivalent to $(V_\frl\ul{M})^{\cG_{L_\infty}}=0$ (see Lemma \ref{lemma.fix}),  so Theorem \ref{thm.intro} follows from the next two assertions: 
\begin{itemize}
\item[(1)] $\prod_{\frl\neq\frp}(\wc V_\frl\ul{M}/\wc T_\frl\ul{M})^{\cG_{L_\infty}}$ is finite, and
\item[(2)] $(\wc V_\frp\ul{M})^{\cG_{L_\infty}}=0$. 
\end{itemize} 
As in previous works on abelian varieties over $p$-adic fields, (1) is easier than (2); in fact, it can be proved under mild conditions than Theorem \ref{thm.intro}; see Proposition \ref{proposition.frl}.

The strategy to prove (2) is as follows.
Assume that $(\wc V_\frp\ul{M})^{\cG_{L_\infty}}$ is non-zero.
Then it has dimension $\leq 2$.
By the theory of crystalline representations (Lemmas \ref{lem.char} and \ref{lemma.key}), replacing $L$ by a finite separable extension, we show that 
$(\wc V_\frp\ul{M})^{\cG_{L_\infty}}$ is a direct sum of $Q_\frp(n)$.
By Condition \ref{condition.main}, it follows that $\ul{M}$ has at least one integral weight; however,  $\ul{M}$ has no integral weights by assumption. 
Thus we get $(\wc V_\frp\ul{M})^{\cG_{L_\infty}}=0$.

\begin{remark}
The reason why we assume $\ul M$ is of rank two is that    the key lemma (Lemma \ref{lemma.key}) can be applied to only two-dimensional crystalline representations, 
although the original of this lemma (\cite[Lemma 3.5]{KT13} and \cite[Lemma 2.5]{Oze20}) holds without assumptions on dimension. 
This restriction comes from 
 the absence of an analogue of $p$-adic logarithmic function providing some useful facts (e.g.,\ \cite[Lemma 2.4]{Oze20}).
So, at this time, it is not easy to extend Theorem \ref{thm.intro} to higher-rank $A$-motives.

On the other hand, 
if $\ul M$ has either rank equal to  $p$ or integral weights, then we can construct examples of $\ul{M}$ having infinite $L_\infty$-valued torsion; see Examples \ref{ex.C1} and \ref{ex.C2}.
Thus it is essential in  Theorem \ref{thm.intro} to assume that $p\neq 2$ and $\ul M$ has no integral weights.
\end{remark}

Abelian Anderson $A$-modules $\ul{G}$ over $L$ can be embedded into the category of $A$-motives over $L$ by a fully faithful functor $\ul G\mapsto \ul{M}(\ul G)$ and hence 
 Theorem \ref{thm.intro} implies  finiteness of $L_\infty$-valued torsion points $\ul{G}(L_\infty)_\tor$:

\begin{corollary}[{$=$ Corollary \ref{corollary.Amod}}]\label{int.cor}
Let $\ul{G}$ be an abelian Anderson $A$-module of rank two over $L$ with good reduction and $z\in Q$ a uniformizing parameter at $\frp$ satisfying  Condition \ref{condition.main}.
If $p\neq 2$ and $\ul{M}(\ul{G})$ has no integral weights, then 
$\ul{G}(L_\infty)_{\tor}$ is finite.
In particular, if $p\neq 2$ and $\ul{G}$ is a Drinfeld $A$-module of rank two with good reduction, then $\ul G(L_\infty)_\tor$ is finite. 
\end{corollary}

\begin{remark}
It is easy to see  that $\ul G(L)_\tor$ is also finite if $\ul G$ is  as in Corollary \ref{int.cor}.
This gives a higher-dimensional generalization of Poonen's result \cite[\S 2, Proposition 1]{Poo97} for Drinfeld $A$-modules.
\end{remark}

As an application of the above, we study torsion points of Drinfeld $A$-modules in a global setting. 
For each abelian variety $\cA$ over a number field $K/\bQ$, Ribet shown in \cite[Appendix, Theorem]{KL81} that 
the torsion subgroup 
$\cA(K(\mu_\infty))_\tor$ is finite, where $\mu_\infty$ is the set of all roots of unity. 
When $Q=\bF_q(t)$ and $A=\bF_q[t]$ with $p\neq 2$, we prove an analogue of Ribet's theorem (Theorem \ref{theorem.Ribet}) for rank-two Drinfeld $A$-modules over global function fields.
In this theorem, $\mu_\infty$ is replaced by torsion points of  the \textit{Carlitz module}.

\subsection{Organization}
In \S 1, we have prepared notation and stated our theorem.
In \S 2, we give a preliminary for $A$-motives and abelian Anderson $A$-modules.
In \S 3, we recall basics on equal characteristic crystalline representations and then study weights of them.
In \S 4, showing the key lemma (Lemma \ref{lemma.key}), we prove the main results.
In \S 5, we study torsion points of Drinfeld $A$-modules in a global setting. 

\section{$A$-motives and abelian Anderson $A$-modules}

In this section, we review definitions and properties of $A$-motives and abelian Anderson $A$-modules.
We refer the details to 
\cite{Gaz24, Har19, HJ20}.

\subsection{Definitions}\label{subsec.Defi}
We set $A_L:=A\otimes_{\bF_q} L$, which is a Dedekind domain.
Denote by $\sigma$ the endomorphism of $A_L$ determined by $\sig(a\otimes \lam)=a\otimes \lam^q$ for any $a\in A$ and $\lam\in L$.
For an $A_L$-module $M$, we put $\sig\ast M:=M\otimes_{A_L, \sig}A_L$.
Let $J_L \subset A_L$ be the ideal generated by $\{a\otimes 1-1\otimes \gamma(a) : a\in A\}$. 
Then $\gamma$ can be recovered as the homomorphism  $A\ra A_L/J_L\cong L$.

Since $J_L$ is invertible (see \cite[\S 1]{Har19}), for each positive integer $n$,  the $A_L$-module $J_L^{-n}$ makes sense and we have canonical inclusions $J_L^{-n} \hra J_L^{-(n+1)}$.
We consider the direct limit
\[
A_L[J_L^{-1}]:=\varinjlim_n A_L\ot_{A_L}J_L^{-n}
\]
taken over non-negative integers $n$.
All $J_L^{-n}$ are flat and direct limits preserve exact sequences, so the natural map $A_L\ra A_L[J_L^{-1}]$ is also flat.
For each $A_L$-module $M$, we put 
\[
M[J_L^{-1}]:=M\ot_{A_L}A_L[J_L^{-1}].
\]

\begin{definition}\label{def.Amot}
An \textit{$A$-motive of rank $r$ over $L$} is a pair $\ul{M}=(M, \tau_M)$ consisting of a projective $A_L$-module $M$ of constant rank $r$ and an isomorphism 
\[
\tau_M\colon \sig\ast M[J_L^{-1}]\overset{\sim}{\lra}M[J_L^{-1}].
\]
We put $\rk \ul{M}:=r$.
We say that $\ul{M}$ is \textit{effective} if $\tau_{M}(\sig\ast M)\subset M$.
Then $J_L^d\cdot M/\tau_M(\sig\ast M)=0$ for some $d\geq 0$ and so $M/\tau_M(\sig\ast M)$ is a finite-dimensional $L$-vector space.
Define the \textit{dimension} $\dim\ul{M}$ of an effective $A$-motive $\ul{M}$ to be the $L$-dimension of $M/\tau_M(\sig\ast M)$. 

A \textit{morphism of $A$-motives} $f\colon \ul{M}\to \ul{N}$ is an morphism of $A_L$-modules $f\colon M\to N$ such that $f\circ \tau_M=\tau_N\circ \sig\ast f.$

In general, for arbitrary $\bF_q$-algebra $S$ equipped with a (not necessarily injective) $\bF_q$-algebra homomorphism $A\to S$, one can define \textit{$A$-motives over $S$} in the same way; see \cite[\S 2]{Har19} for details.
\end{definition}

\begin{example}
Under the isomorphism 
$
\sig\ast A_L \overset{\sim}{\lra} A_L 
$ given by $c\ot 1\mapsto \sig(c)$, 
we may identify $\sig\ot\id$ with $\id\colon A_L\to A_L$.
Then the pair $\ul{A_L}=(A_L, \id)$ is an effective $A$-motive of rank one and  dimension zero. 
\end{example}

\begin{definition}
Let $\ul{M}$ and $\ul{N}$ be $A$-motives over $L$.
For each $L$-algebra $S$, the \textit{base change} 
$
\ul{M}_S:=(M\otimes_{A_L}A_S, \tau_{M}\ot \id)
$
is an $A$-motive over $S$.
 The \textit{direct sum} and \textit{tensor product} of $\ul{M}$ and $\ul{N}$ are defined as the $A$-motives 
$
\ul{M}\oplus \ul{N}:=(M\oplus N, \tau_M\oplus \tau_N)
$ and 
$
\ul{M}\ot\ul{N}:=(M\ot_{A_F}N, \tau_M\ot\tau_N)
$.
It is easy to see that 
$\rk \ul{M}_S=\rk \ul{M}, 
\rk(\ul{M}\oplus \ul{N})=\rk \ul{M}+\rk\ul{N}
$, and 
$
\rk(\ul{M}\ot\ul{N})=(\rk\ul{M})\cdot(\rk \ul{M})
$.

Under the identification
$\Hom_{A_L}(M,N)[J_L^{-1}] = \Hom_{A_L[J_L^{-1}]}(M[J_L^{-1}], N[J_L^{-1}])$,
we
 define the \textit{internal hom} $\iHom(\ul{M}, \ul{N})=(H,\tau_H)$ by the $A$-motive consisting of the $A_L$-module $H=\Hom_{A_L}(M,N)$ and the isomorphism
\[\tau_H\colon 
 \sig\ast H[J_L^{-1}]= \Hom_{A_L[J_L^{-1}]}(\sig\ast M[J_L^{-1}], \sig\ast N[J_L^{-1}]) \overset{\sim}{\lra} H[J_L^{-1}] 
\]
given by $ h\mapsto \tau_N\circ h \circ \tau_M^{-1}$.
We define the \textit{dual} of $\ul{M}$ by 
\[
\ul{M}^{\vee}:=\iHom(\ul{M}, \ul{A_L}).
\]
We have $\rk\iHom(\ul{M},\ul{N})=(\rk\ul{M})\cdot(\rk\ul{N})$ and hence $\rk\ul{M}=\rk\ul{M}^{\vee}$.
For  each integer $n$, the \textit{tensor power} $\ul{M}^{\ot n}$ of $\ul{M}$ is defined by $\ul{M}^{\ot 0}:=\ul{A_L}$,  $\ul{M}^{\ot n}:=\ul{M}^{\ot(n-1)}\ot \ul{M}$ if $n>0$, and $\ul{M}^{\ot n}:=(\ul{M}^\vee)^{\ot (-n)}$ if $n<0$.
\end{definition}

Now we recall the construction of Galois representations attached to $A$-motives.
Put
\[
A_{\frl, L}:=A_\frl\wh{\ot}_{\bF_q}L=\varprojlim_n(A/\frl^n)\ot_{\bF_q}L
\]
and $M_\frl:=M\ot_{A_L}A_{\frl, L}$.
Then $\tau_M$ induces an isomorphism
\[
\tau_M\colon \sig\ast M_\frl({L^\sep}) \overset{\sim}{\lra} M_\frl({L^\sep}),
\]
where $M_\frl(L^\sep)=M_\frl\ot_{L}L^\sep$.
For each $m\in M_\frl({L^\sep})$, 
we set $\sig\ast m:=m\ot 1\in \sig\ast M_\frl({L^\sep})$.
Then the \textit{$\frl$-adic realization} of $\ul{M}$ is defined by 
\[
\wc T_\frl\ul{M}:=\ul{M}_\frl(L^\sep)^\tau:=\{m\in M_\frl({L^\sep}) : \tau_M({\sig\ast m})=m\}, 
\]
which is free of rank $\rk \ul{M}$ over $A_\frl$ with a continuous $\cG_L$-action. 
It is known that the functor $\ul{M} \mapsto \wc T_\frl\ul{M}$ is exact.
We also define the \textit{rational $\frl$-adic realization} of $\ul{M}$ by 
\[
\wc V_\frl\ul{M}:=\wc T_\frl\ul{M}\ot_{A_\frl}Q_\frl.
\]
By construction, 
we have $\cG_L$-equivariant isomorphisms
\[
\wc T_\frl (\ul{M}\oplus\ul{N})\cong \wc T_\frl\ul{M}\oplus \wc T_\frl\ul{N},\es\es \wc T_\frl(\ul{M}\ot\ul{N})\cong \wc T_\frl\ul{M}\ot_{A_\frl}\wc T_\frl\ul{N},
\]
and 
\[
 \wc T_\frl\ul{M}^\vee \cong \Hom_{A_\frl}(\wc T_\frl\ul{M}, A_\frl).
\]

The (rational) $\frl$-adic realization of $\ul{M}$ can be considered in terms of \textit{\'etale finite $\bF_q$-shtukas} (see \cite[\S 4 and \S 6]{Har19} for details).
Let $\fra$ be a non-zero ideal of $A$.
Then $\tau_M$ induces an isomorphism 
\[
\tau_{M/\fra M} \colon \sig\ast M/\fra M\ot_{A_L}A_{L^\sep} \overset{\sim}{\lra} M/\fra M\ot_{A_L}A_{L^\sep}
\]
and the $\tau$-invariants
\[
(\ul{M}/\fra\ul{M})(L^\sep)^\tau:=\{\ol{m}\in M/\fra M\ot_{A_L}A_{L^\sep} \colon \tau_{M/\fra M}(\sig\ast \ol{m})=\ol m\}
\]
is a free $A/\fra$-module of rank $\rk \ul{M}$ with a continuous $\cG_L$-action.
Now the natural surjections $M/\frl^{n+1}M \thra M/\frl^n M$ give rise to  an inverse system 
$\{(\ul{M}/\frl^n\ul{M})(L^\sep)^\tau\}$ and we get a $\cG_L$-equivariant isomorphism
\[
\wc T_\frl\ul{M} \overset{\sim}{\lra} \varprojlim_n (\ul{M}/\frl^n\ul{M})(L^\sep)^\tau.
\]
In particular, for each separable extension $L'/L$, we see that 
\[
\wc V_\frl\ul{M}^{\cG_{L'}}\cong \varprojlim_n(\ul{M}/\frl^n\ul{M})(L')^\tau\ot_{A_\frl}Q_\frl.
\]

Next, let $\bG_{a,L}=\Spec L[X]$ be the additive group over $L$.

\begin{definition}\label{def.And}
Let $r$ and $d$ be positive integers.
An {\it abelian Anderson $A$-module over $L$ of rank $r$ and dimension $d$} is a pair $\ul G=(G, \phi)$ consisting of an affine $\bF_q$-module scheme $G= \bG_{a,L}^d$ over $L$ and an $\bF_q$-algebra homomorphism $\phi\col A\ra \End_{\SSC L{\mbox{\tiny-groups}}, \bF_q{\mbox{\tiny-lin}}}(G); a\mapsto \phi_a$ such that  
\begin{itemize}
\item $(\Lie \phi_a-a)^d=0$ on $\Lie G$ for any $a\in A$,
\item the set of $\bF_q$-liner homomorphisms of $L$-group schemes
\[
M(\ul G):=\Hom_{\SSC{L\mbox{\tiny-groups},\ \bF_q\mbox{\tiny -lin.}}}(G,\bG_{a,L})
\]
  is a projective $A_L$-module of rank $r$ via $(a\ot \lam)\cdot m:= \lam\circ m\circ \phi_a$ for $m\in M(\ul{G})$, $a\in A$, and $\lam \in L$.
\end{itemize}
We write $\rk \ul{G}$ for the rank of $\ul{G}$.

Note that if $A=\bF_q[t]$ and $L$ is perfect,  then abelian Anderson $A$-modules coincide with \textit{abelian $t$-modules} in the sense of Anderson \cite[\S 1.1]{And86}.

One-dimensional abelian Anderson $A$-modules $\ul G$ over $L$ are called \textit{Drinfeld $A$-modules} over $L$.
To describe this, let $\tau$ be the relative $q$-th Frobenius of $\bG_{a, L}=\Spec L[X]$ given by $\tau(X)=X^q$ and consider the non-commutative polynomial ring $L\{\tau\}$ over $L$ in variable $\tau$ whose multiplication is defined so that $(b\tau^i)\cdot(c\tau^j)=bc^{q^i}\tau^{i+j}$ for $b,c\in L$.
This ring is isomorphic to $\End_{\SSC L{\mbox{\tiny-groups}}, \bF_q{\mbox{\tiny-lin}}}(\bG_{a,L})$ and  so $\ul G$ is determined by an $\bF_q$-algebra homomorphism
\[
\phi\colon A\to L\{\tau\}\ ;\  a\mapsto \gamma(a)+(\text{higher terms})
\]
such that $\Im \phi\not\subset L$.
If $\ul G$ is of rank $r$, then for each $a\in A$, we have 
\[
\phi_a=\gamma(a)+\cdots+\lambda_a \tau^{r\deg a}\es\es \text{with}\es\es \lambda_a\neq 0,
\]
where $\deg a=\dim_{\bF_q}(A/aA)$.
\end{definition} 

Let $\ul{G}=(G,\phi)$ be an abelian Anderson $A$-module over $L$.
Then for each $L$-algebra $S$, the set of $S$-valued points $G(S)$ becomes an $A$-module via $\phi$.
We write $\ul{G}(S)$ for this $A$-module.
We denote the torsion submodule of it by 
\[
\ul{G}(S)_\tor:=\{\lambda \in G(S) \colon \phi_a(\lambda)=0\ \text{for some non-zero}\ a\in A\}.
\] 
The \textit{$\frl$-adic Tate module} and the \textit{rational $\frl$-adic Tate module} of $\ul G$ are defied by 
\[
T_\frl\ul{G}:=\Hom_{A_\frl}(Q_\frl/A_\frl,\ \ul{G}(L^\sep))\es\es \text{and}\es\es V_\frl\ul{G}:=T_\frl\ul{G}\ot_{A_\frl}Q_\frl.
\]
To see the structures of these modules, we consider the \textit{$\fra$-torsion submodule} of $\ul{G}$
  \[
\ul{G}[\fra]:=
\bigcap_{0\neq a\in \fra}\ul{G}[a],
\] 
where $\ul{G}[a]:=\Ker(\phi_a\colon G\to G)$.
Remark that  $\ul{G}[\fra]=\ul{G}[a]$ if $\fra=(a)$. 
It follows by  \cite[Theorem 6.4]{Har19} that 
$\ul G[\fra]$ is a finite (locally) free closed subgroup scheme of $G$ and has a natural $A/\fra$-module structure via $\phi$, and it is \'etale because $\gamma$ is injective.
By \cite[Theorem 6.6]{Har19}, we have an isomorphism of $A/\fra$-modules
\[
\ul{G}[\fra](L^\sep)\overset{\sim}{\lra} (A/\fra)^{\oplus \rk \ul{G}}
\]
and   $\ul G[\fra](L^\sep)$  has a continuous $\cG_L$-action compatible with its $A$-module structure. 
Then there is a $\cG_L$-equivariant isomorphism
\[
\varprojlim_n \ul{G}[\frl^n](L^\sep) \overset{\sim}{\lra} T_\frl\ul{G}
\]
and so $T_\frl\ul G$ is a free $A_\frl$-module of rank $\rk\ul{G}$ with a continuous $\cG_L$-action. 

Let $\tau$ be the relative $q$-Frobenius endomorphism of $\bG_{a,L}=\Spec L[X]$ given by $X\mapsto X^q$.
For the $A_L$-module $M(\ul G)$,  
 define 
 \[
 \tau_{M(\ul G)}\col\sig\ast M(\ul G)\ra M(\ul G)
 \] by $\tau_M(\ul G)(\sig\ast m)=\tau\circ m$ for $m\in M(\ul G)$.
Then  $\ul{M}(\ul G)=(M(\ul G),\tau_{M(\ul G)})$ is an effective $A$-motive over $L$ whose rank and dimension are equal to those of $\ul G$, and the functor 
\[
\ul G \mapsto \ul{M}(\ul G)
\]
 is fully faithful; see \cite[Theorem 1]{And86} and \cite[Theorem 3.5]{Har19}.
 In addition, there is a $\cG_L$-equivariant isomorphism
 \[
 \ul G[\frl^n](L^\sep) 
 \overset{\sim}{\lra} 
 \Hom_{A/\frl^n}((\ul{M}(\ul G)/\frl^n \ul{M}(\ul G))(L^\sep)^\tau, \ \Hom_{\bF_q}({A/\frl^n}, \bF_q)) 
 \]
 for each $n\geq 1$ and hence we have 
 \[
 T_\frl\ul G \cong \Hom_{A_\frl}(\wc T_\frl \ul{M}(\ul G), A_\frl) \cong\wc T_\frl \ul{M}(\ul G)^\vee.
 \]
 Therefore if we put $\ul G[\frl^\infty](L^\sep):=\bigcup_{n=1}^\infty\ul{G}[\frl^n](L^\sep)$, there is a $\cG_L$-equivariant isomorphism
 \[
 \wc V_\frl\ul{M}(\ul G)/\wc T_\frl\ul{M}(\ul G) \cong V_\frl\ul{G}/T_\frl\ul{G}\cong \ul{G}[\frl^\infty](L^\sep).
 \]
 By this, the finiteness of $\ul{G}(L_\infty)_\tor$
 and  $\prod_{\frl}(\wc V_\frl\ul{M}(\ul G)/\wc T_\frl\ul{M}(\ul G) )^{\cG_{L_\infty}}$ are equivalent. 
\subsection{Good reduction and weights of $A$-motives}

Let $\ul M$ be an $A$-motive over $L$.
\begin{definition}
An $A$-motive $\ul \cM$ over $R$ satisfying $\ul{\cM}_L\cong \ul{M}$ is called a \textit{good model} of $\ul M$.
If $\ul M$ has a good model, then it is said to have \textit{good reduction}.
Then we define  the \textit{reduction} of $\ul{M}$ by
\[
\ul{M}_k:=\ul \cM_k=(M_k:=\cM\ot_{A_R}A_k,\ \tau_k:=\tau_{\cM}\ot \id),
\]
which is an $A$-motive over $k$.

An abelian Anderson $A$-module $\ul G$ over $L$ is said to have \textit{good reduction} if the associated $A$-motive $\ul M(\ul G)$ has good reduction.
\end{definition}

Suppose that  $\ul{M}$ has a good model  
$\ul{\cM}$.
For the reduction $\ul{M}_k=(M_k, \tau_k)$, 
since $\sig^{[k:\bF_q]}$ is identity on $k$, there is a canonical isomorphism $\sigma^{[k:\bF_q]*}M_k \cong  M_k$.
Then we have an automorphism
\[
\tau_k^{[k:\bF_q]}\colon M_k[J_k\inv] \overset{\sim}{\lra} M_k[J_k\inv],
\]
where $\tau_k^{[k:\bF_q]}:=\tau_k\circ \sig\ast\tau_k\circ\cdots \circ \sig^{([k:\bF_q]-1)*}\tau_k$.
Let $Q_k=\mathrm{Frac}(A_k)$ be the field of fractions of $A_k$.
Then the isomorphism $\sig\colon A_k \overset{\sim}{\ra} A_k$
induces $\sigma\colon Q_k \overset{\sim}{\ra} Q_k$.
By construction, $Q_k$ is a Galois extension of $Q_{\bF_q}=Q$, and its Galois group is generated by $\sig$.
We may view $\tau_k^{[k:\bF_q]}$ as an automorphism of the $Q_k$-vector space $M_k\otimes_{A_k}Q_k$.
We define the \textit{characteristic polynomial of $\ul{M}$} by  
\[
P(X; \ul{M}):=\det(X-\tau_k^{[k:\bF_q]}),
\]
which has coefficients in $Q$ by \cite[Lemma 8.1.4]{BP09}.
We fix an algebraic closure $Q^\alg$ of $Q$ and denote by  
\[
|\cdot|\colon Q^\alg \ra \bR_{\geq 0}
\]
 the unique extension of the normalized absolute value of $Q$ with respect to the distinguished place $\infty$.

\begin{definition}\label{definition.wmot}
An $A$-motive $\ul{M}$ over $L$ with good reduction is said to
have \textit{weights $w_1,\ldots, w_r$} if 
$P(X; \ul{M})=\prod_{i=1}^r(X-\alpha_i)$ for some elements $\alpha_i\in Q^\alg$ such that 
\[
|\alpha_i|=(\#k)^{w_i}.
\]
We say that $\ul{M}$ has \textit{no integral weights} if $w_i\notin \bZ$ for any $i$.
\end{definition}

Let $
\rho_{\ul{M}, \frl}\colon \cG_L\ra \Aut_{A_\frl}\wc T_\frl\ul{M}
$ be 
 the associated Galois representation and $\Frob_k\in \cG_k=\Gal(k^\sep/k)$ the geometric Frobenius
 element (i.e., its inverse $\Frob_k^{-1}$ acts on $k^\sep$ by $\lambda\mapsto \lambda^{\#k}$).
If $\frl\neq \frp$, then the N\'eron-Ogg-Shafarevich-type criterion (cf.\ \cite[Theorem 1.1]{Gar02} and \cite[Proposition 4.49]{Gaz24}) implies that $\wc T_\frl\ul{M}$ is unramified, that is, the action of the inertia subgroup $\cI_L\subset \cG_L$ is trivial.
Hence $\rho_{\ul{M},\frl}$ factors through $\cG_k \cong \cG_L/\cI_L$, so $\rho_{\ul{M},\frl}(\Frob_k)$ is well-defined.

\begin{lemma}[{cf. \cite[Proposition 2.3.36]{HJ20}}]\label{lemma.Frob}
Under the above notation, we have 
\[
P(X; \ul{M})=\det(X-\rho_{\ul{M},\frl}(\Frob_k))
\]
for each $\frl\neq \frp$.
\end{lemma}

\begin{proof}
There is an isomorphism $\wc T_\frl\ul{M} \overset{\sim}{\lra} \wc T_\frl\ul{M}_k$ as $A_\frl$-modules with $\cG_k$-action. 
For $M_{k,\frl}=M_k\otimes_{A_k}{A_{\frl, k}}$ and $M_{k,\frl}(k^\sep)=M_{k,\frl}\ot_{A_{\frl,k}}A_{\frl, k^\sep}$, we have a canonical isomorphism
\[
h\colon \wc T_\frl\ul{M}_k \otimes_{A_\frl}A_{\frl,k^\sep} \overset{\sim}{\lra} M_{k,\frl}(k^\sep)
\]
induced by the inclusion $\wc T_\frl\ul{M}_k= M_{k,\frl}(k^\sep)^\tau \hra  M_{k,\frl}(k^\sep)$.
Then we have the following commutative diagram.
\[
\xymatrix@C=50pt@R=30pt{
 M_{k,\frl}(k^\sep) \ar@{=}[r] 
& \sig^{[k:\bF_q]*} M_{k,\frl}(k^\sep) \ar[r]^-{\tau_k^{[k:\bF_q]}}_-{ \cong }
&  M_{k,\frl}(k^\sep)
\\
\wc T_\frl\ul{M}_k \otimes_{A_\frl}A_{\frl,k^\sep} \ar[u]_-{ \cong }^-{h}
&\wc T_\frl\ul{M}_k \otimes_{A_\frl}A_{\frl,k^\sep} \ar[l]_-{ \cong }^-{\SSC \rho_{\ul{M},\frl}(\Frob_k\inv)\otimes\id}  \ar[u]_-{ \cong }^-{(\Frob_k\inv)^*h} \ar@{=}[r] 
&\wc T_\frl\ul{M}_k \otimes_{A_\frl}A_{\frl,k^\sep}  \ar[u]_-{ \cong }^h
}
\]
This implies that $\rho_{\ul{M},\frl}(\Frob_k)=h^{-1}\circ \tau_k^{[k:\bF_q]}\circ h$ on $\wc T_\frl\ul{M}$ and so we get the conclusion.
\end{proof}

\begin{example}
Suppose that  a Drinfeld $A$-module $\ul G$ over $L$ has good reduction, that is equivalent to the existence of an element $\lambda\in L$ such that 
$\lambda^{-1}\phi_a\lambda \in R\{\tau\}$ and the leading coefficient of $\lambda^{-1}\phi_a\lambda$ belongs to $R^\times$ for each $a\in A$. 
By \cite[Theorem 4.12.8]{Gos96}, if $\ul G$ is of rank $r$, then each eigenvalue $\alpha$ of $\Frob_{k}^{-1}$ on $T_\frl\ul G$ with $\frl\neq \frp$ satisfies 
\[
|\alpha|=(\# k)^{\frac{1}{r}}.
\]
Since $\wc T_\frl \ul M(\ul G)\cong \Hom_{A_\frl}(T_\frl \ul G, A_\frl)$, 
the $A$-motive $\ul M(\ul G)$ has weights $\frac{1}{r}, \ldots, \frac{1}{r}$ with multiplicity $r$.
\end{example}
%
\section{Equal characteristic crystalline representations}
We shall give a brief overview of Hodge-Pink theory and equal characteristic crystalline representations of $\cG_L$ needed for our study (cf.\ \cite{GL11, Har11, HK20}).

\subsection{Local shtukas}
Recall that we have put $\wh q=\#\bF_\frp$.
We consider the ring $R\wh\ot_{\bF_\frp}A_\frp \cong  R\zb$ and denote by $\wh\sig$ the ring endomorphism 
of $R\zb$ with $\wh\sig(z)=z$ and $\wh\sig(\lambda)=\lambda^{\wh q}$ for any $\lambda\in R$.
Then $\wh \sig$ is extended to ring endomorphisms of $L\zb$ and $L^\sep\zb$, which are also denoted by $\wh \sig$.
The absolute Galois group $\cG_L$ of $L$ acts on $L^\sep\zb$ by 
\[
g\cdot \sum_{n=0}^\infty \lambda_nz^n:=\sum_{n=0}^\infty g(\lambda_n)z^n
\]
for any $g\in \cG_L$.
For each $R\zb$-module $\wh M$, we set $\wh\sig\ast \wh M:=\wh M\otimes_{R\zb, \wh\sig}R\zb$.

\begin{definition}
A \textit{local shtuka of rank $r$} over $R$ is a pair $\ul{\wh M}=(\wh M, \tau_{\wh M})$ consisting of a free $R\zb$-module $\wh M$ of rank $r$ and an isomorphism $\tau_{\wh M}\colon \wh\sig\ast \wh M[\zzf]\overset{\sim}{\lra} \wh M[\zzf]$.
We write $\rk \ul{\wh M}$ for the rank of $\ul{\wh M}$.

A \textit{morphism} of local shtukas $f\colon \ul{\wh M}\ra \ul{\wh N}$ over $R$ is a morphism of $R\zb$-modules $f\colon \wh M\ra \wh N$ such that 
$\tau_{\wh N}\circ \wh\sig\ast f=f\circ \tau_{\wh M}$. 
We write $\Hom_R(\ul{\wh M}, \ul{\wh N})$ for the $A_\frp$-module of morphisms $f\colon \ul{\wh M}\ra \ul{\wh N}$, and set $\End_R(\ul{\wh M}):=\Hom_R(\ul{\wh M}, \ul{\wh M})$.

A \textit{quasi-morphism} of local shtukas $f\colon \ul{\wh M}\ra \ul{\wh N}$ over $R$ is a morphism of $R\zb[\zf]$-modules $f\colon \wh M[\zf]\ra \wh N[\zf]$ such that 
$\tau_{\wh N}\circ \wh\sig\ast f=f\circ \tau_{\wh M}$. 
It is called a \textit{quasi-isogeny} if it is an isomorphism of $R\zb[\zf]$-modules.
Let us denote by $\QEnd_R(\ul{\wh M}, \ul{\wh N})$ the $Q_\frp$-vector space of quasi-morphisms and set $\QEnd_R(\ul{\wh M}):=\QHom_R(\ul{\wh M}, \ul{\wh M})$.

Similar as $A$-motives, one can consider the \textit{direct sums} and \textit{tensor products} of local shtukas.
\end{definition}

By \cite[Corollary 3.4.5]{HK20}, the $A_\frp$-module $\Hom_R(\ul{\wh M}, \ul{\wh N})$ is free of rank at most $\rk \ul{\wh M}\cdot \rk\ul{\wh N}$.
We also have $\QHom_{R}(\ul{\wh M}, \ul{\wh N})=\Hom_R(\ul{\wh M}, \ul{\wh N})\otimes_{A_\frp}Q_\frp$.
If there is a quasi-isogeny $f\colon \ul{\wh M}\ra \ul{\wh N}$, then $g\mapsto f\circ g \circ f^{-1}$ defines an isomorphism of $Q_\frp$-algebras 
$\QEnd_R(\ul{\wh M})\overset{\sim}{\lra} \QEnd_R(\ul{\wh N})$.
We denote by 
\[
\mathsf{QLS}(R)
\]
the category of local shtukas over $R$ with quasi-morphisms, so that 
isomorphisms in this category are quasi-isogenies.

Let $\wh{\ul M}=(\wh M,\tau_{\wh M})$ be a local shtuka over $R$.
Then $\tau_{\wh M}$ induces an isomorphism 
\[
\tau_{\wh M} \col \hsig\ast \wh M \ot_{R\zb}L^\sep\zb \lrai \wh M\ot_{R\zb}L^\sep\zb
\]
because $\zz$ is invertible in $L^\sep\zb$.
Now $\cG_L$ acts on 
$\wh M \ot_{R\zb}L^\sep\zb$ via 
\[
g\cdot (m\ot \mu(z))=m\ot g\cdot \mu(z)
\]
 for $m\in M$ and $\mu(z)\in L^\sep\zb$.
We denote by 
\[
\wh\sig\ast_{\wh M}\colon \wh M\ot_{R\zb}L^\sep\zb \to \wh\sig\ast\wh M\ot_{R\zb}L^\sep\zb
\]
the natural map given by  $\wh\sig\ast_{\wh M}m=m\ot 1$ for each $m\in \wh M\ot_{R\zb}L^\sep\zb$. 
Define 
the {\it $\frp$-adic realization} of $\ul{\wh M}$ to be 
\[  
\wc T_\frp \wh{\ul M}:=(\wh M \ot_{R\zb}L^\sep\zb)^{\wh\tau}:=
\{ m \in \wh M \ot_{R\zb}L^\sep\zb  :  \tau_{\wh M}(\hsig\ast_{\wh M} m)=m\},
\] 
which 
 is a free $A_\frp$-module of rank equal to  $\rk \wh{\ul M}$ with a continuous $\cG_L$-action.
We also define the {\it rational $\frp$-adic realization} of  $\wh{\ul M}$ by 
\[
\wc V_\frp \wh{\ul M}:= \wc T_\frp \wh{\ul M} \ot_{A_\frp}Q_\frp.
\]

By the similar argument in the proof of \cite[Proposition 3.4.22]{HK20}, we get the following.

\begin{lemma}\label{lemma.sub}
For any $\cG_L$-stable $Q_\frp$-subspace $V \subset \wc V_\frp\ul {\wh M}$, there is a local shtuka $\ul{\wh M'}$ over $R$ such that $V\cong\wc V_\frp\ul{\wh M'}$.
\end{lemma}

\begin{proof}
By {\cite[Proposition 3.4.2]{HK20}},
the inclusion $\wc T_\frp\ul{\wh M} \hra \wh M\ot_{R\zb}L^\sep\zb$ defines a canonical isomorphism of $L^\sep\zb$-modules
\begin{equation*}
h\colon \wc T_\frp\ul{\wh M}\ot_{A_\frp}L^\sep\zb \overset{\sim}{\lra} \wh M\ot_{R\zb}L^\sep\zb
\end{equation*}
which is functorial in $\ul {\wh M}$ and $\cG_L$- and $\wh \tau$-equivariant, where on the left module 
$\cG_L$ acts on both factors and $\wh \tau=\id\ot\wh\sig$, and on the right module $\cG_L$ acts only on $L^\sep\zb$ and $\wh \tau=(\tau_{\wh M}\circ \wh \sig_{\wh M}^*)\ot\wh\sig$.
In particular, one can recover $\wh M\ot_{R\zb}L\zb=(\wc T_\frp\ul{\wh M}\ot_{A_\frp}L^\sep\zb)^{\cG_L}$.
Under the natural isomorphism
\begin{equation}\label{eq.iden}
\wh\sig\ast L^\sep\zb=L^\sep\zb\ot_{L^\sep\zb, \wh \sig}L^\sep\zb \overset{\sim}{\lra} L^\sep\zb;\  x\ot  y\mapsto \wh\sig(x)y,
\end{equation}
 we see that $h$ induces the following commutative diagram.
\begin{equation}\label{diag.h}
\vcenter{
\xymatrix@C=40pt@R=30pt{
\wc T_\frp\ul{\wh M}\ot_{A_\frp}L^\sep\zb\ar@{=}[d] & \ar[l]^{\cong}_{\id\ot(\ref{eq.iden})} 
 \wc T_\frp\ul{\wh M}\ot_{A_\frp}\wh\sig\ast L^\sep\zb \ar[r]_{\cong}^{\wh\sig\ast h}&
 \wh\sig\ast\wh M\ot_{R\zb}L^\sep\zb \ar[d]_{\cong}^{\tau_{\wh M}}
\\
\wc T_\frp\ul{\wh M}\ot_{A_\frp}L^\sep\zb \ar[rr]_{\cong}^{h} & & \wh M\ot_{R\zb}L^\sep\zb
}
}
\end{equation}

Put $T:=V\cap \wc T_\frp\ul{\wh M}$, which is a $\cG_L$-stable $A_\frp$-lattice of $V$.
We define an $L\zb$-submodule $N$ of $\wh M\ot_{R\zb}L\zb$ by 
\[
N=h\left((T\ot_{A_\frp}L^\sep\zb)^{\cG_L}\right).
\]
It is free over $L\zb$, so the elementary divisor theorem implies that 
$N[\zf]$ is a direct summand of $\wh M[\zf]\ot_{R\zb}L\zb$.
We set $\wh M':=N\cap \wh M$.
Then  $\wh M'[\zf]\ot_{R\zb}L\zb=N[\zf]$.
By construction, $\wh M'$ is a finitely generated torsion free $R\zb$-module satisfying
$\wh M'=\wh M'[\zf]\cap (\wh M'\ot_{R\zb}L\zb)$ and $\tau_{\wh M}(\wh\sig\ast M'[\zzf])=M'[\zzf]$ by the diagram (\ref{diag.h}).
It follows by \cite[Lemma 3.4.23]{HK20} that 
$\ul{\wh M'}=(\wh M', \tau_{\wh M'})$ is a local shtuka over $R$, where $\tau_{\wh M'}$ is the restriction of $\tau_{\wh M}$ to $\wh\sig\ast M'[\zzf]$.

We claim that $\wc T_\frp\ul{\wh M'}=T$.
To prove this, it suffices to show that $\wh M'\ot_{R\zb}L\zb=N$.
By construction,  we have $\wh M'=\wh M'[\zf]\cap N$ and $N[\zf]=\wh M'[\zf]\ot_{R\zb}L\zb$, so we get an exact sequence
\[
0\to \wh M'\to \wh M'[{\textstyle\frac{1}{z}}] \to N[{\textstyle\zf}]/N.
\]
Since we have an isomorphism 
\begin{align*}
(N[{\textstyle \zf}]/N)\ot_{R\zb}L\zb
&\cong \varinjlim_r \left(\overset{}{z^{-r}N/N \ot_{R\zb/(z^r)}L\zb/(z^r)}\right) \\
&\cong \varinjlim_r \left(z^{-r}N/N\right)[{\textstyle \frac{1}{\zeta}}]\\
&\cong N[{\textstyle \frac{1}{z}}]/N, 
\end{align*}
we get the following exact sequence 
\[
0\to \wh M'\ot_{R\zb}L\zb \to N[{\textstyle \zf}] \to N[{\textstyle \zf}]/N.
\]
This implies the claim.
Consequently, we see that $\wc V_\frp\ul{\wh M'}=T\ot_{A_\frp}Q_\frp \cong V$.
\end{proof}

Let $\mathsf{Rep}_{Q_\frp}(\cG_L)$ be the category of finite-dimensional $Q_\frp$-representations of $\cG_L$.
It follows by \cite[Theorem 3.4.20]{HK20} that $\ul{\wh M}\mapsto \wc V_\frp\ul{\wh M}$ defines an exact tensor fully faithful functor
\[
\wc V_\frp \colon \mathsf{QLS}(R) \lra \mathsf{Rep}_{Q_\frp}(\cG_L).
\]
We write 
\[
\mathsf{Rep}_{Q_\frp}(\cG_L)^\crys
\]
for the full subcategory of $\mathsf{Rep}_{Q_\frp}(\cG_L)$ whose objects are the essential image of $\wc V_\frp$.
We refer objects of $\mathsf{Rep}_{Q_\frp}(\cG_L)^\crys$ to \textit{$($$z$-adic$)$ crystalline representations} of $\cG_L$.

\begin{example}[{cf.\ \cite[Example 3.4.11]{HK20}}]\label{example.cyclotomic}
Define the local shtuka $\ul{\mathbbm{1}}(1)$ over $R$ of rank one by 
\[
\ul{\mathbbm{1}}(1)=(R\zb\cdot\mathbf{e},\es \wh\sig\ast \mathbf{e}\mapsto(\zz)^{-1}\mathbf{e}).
\]
For each $n\in \bZ_{\geq 0}$, we take elements $\ell_n\in F_\frp^\sep$ as solutions to the equations 
$\ell_0^{\wh q-1}=-\zeta$ and $\ell_n^{\wh q}+\zeta \ell_n=\ell_{n-1}$.
Then the power series $\ell^+:=\sum_{n=0}^\infty \ell_nz^n\in F_\frp^\sep\zb \subset L^\sep\zb$ satisfies $\wh\sig(\ell^+)=(\zz)\cdot \ell^+$.
Although $\ell^+$ depends on the choice of the $\ell_n$, if we choose different $\wt\ell_n$ and define a power series $\wt\ell^+$, we have 
$\wt\ell^+=u\ell^+$ for some unit $u\in A_\frp^\times$.
Indeed, $u=\wt\ell^+/\ell^+$ satisfies 
\[
\wh\sig(u)=\frac{\wh\sig(\wt \ell^+)}{\wh\sig(\ell^+)}=
\frac{(\zz)\cdot\wt \ell^+}{(\zz)\cdot\ell^+}
=
\frac{\wt \ell^+}{\ell^+}=u
\]
and so $u\in (L^\sep\zb^\times)^{\wh\sig=\id}=\bF_\frp\zb^\times= A_\frp^\times$.
Then we have 
$\wc T_\frp \ul{\mathbbm{1}}(1)=A_\frp\cdot \ell^+$.
Therefore the $\cG_L$-action on $\wc T_\frp \ul{\mathbbm{1}}(1)$ is given by the $z$-adic cyclotomic character $\chi_z$, 
so that $\wc V_\frp\ul{\mathbbm{1}}(1)=Q_\frp(1)$. 
\end{example}

\begin{example}[{cf.\ \cite[Example 3.2.2]{HK20}}]\label{ex.motiveshtuka}
Let $\ul M$ be an $A$-motive over $L$ 
with good reduction and $\ul \cM=(\cM, \tau_\cM)$ a good model  of $\ul M$. 
Now we consider the $\frp$-adic completion $A_{\frp,R}:=A_\frp\wh\ot_{\bF_q}R$ of $A_R$ and set 
$\ul \cM\ot{A_{\frp,R}}:=(\cM\ot_{A_R}A_{\frp,R},\tau_\cM\ot\id$).
Then we get the  {\it local shtuka associated with $\ul M$} denoted by $\wh {\ul M}_\frp(\ul M)$ as follows.

\begin{itemize}
\setlength{\leftskip}{-10pt}
\item Assume $\deg \frp=[\bF_\frp:\bF_q]=1$, so that $\bF_\frp=\bF_q$, $\wh q=q$, and $\wh \sig=\sig$.
 Then $A_{\frp,R}=R\zb$ and $J_R\cdot A_{\frp,R}=(\zz)$.
Hence $\ul\cM\ot A_{\frp,R}$ itself becomes a local shtuka over $R$.
 We set 
$\wh {\ul M}_\frp(\ul M):=\ul\cM\ot A_{\frp, R}$. 

\item On the other hand, let us assume   $\deg \frp>1$.
For each $0 \leq i \leq \deg\frp-1$, writing $\fra_i\subset A_{\frp, R}$ for the ideal generated by $\{b\ot1-1\ot \gamma(b)^{q^i} : b \in \bF_\frp\}$,
 we have an identification 
\[
A_{\frp, R}=\prod_{i=0}^{\deg\frp-1} A_{\frp, R}/\fra_i
\]
whose factors have  canonical isomorphisms $ A_{\frp, R}/\fra_i \cong R\zb$.
In addition, the factors are cyclically permuted by $\sig$ since $\sig(\fra_i)=\fra_{i+1}$, and hence $\hsig=\sig^{\deg\frp}$ stabilizes each factor.
Here it follows that the  ideal $J_R$ decomposes as $J_R\cdot A_{\frp, R}/\fra_0=(\zz)$ and $J_R\cdot A_{\frp, R}/\fra_i=(1)$ for $i\neq 0$.
Considering $\tau_\cM^{\deg\frp}:=\tau_\cM \circ \sig\ast\tau_\cM \circ \cdots \circ \sig^{(\deg\frp-1)*}\tau_\cM$, we get the local shtuka 
\[
\ul{\wh M}_\frp(\ul M):=(\cM\ot_{A_R}\left(A_{\frp,R}/\fra_0\right), \tau_{\cM}^{\deg\frp}\ot \id)
\]
over $R$.
This definition coincides with the before one when $\deg\frp=1$.
\end{itemize}
By construction, $\ul{\wh M}_\frp(\ul M)$ has the same rank of $\ul{M}$ and there is a canonical $\cG_L$-equivariant isomorphism
\[
\wc T_\frp\ul{M} \overset{\sim}{\lra} \wc T_\frp \ul{\wh M}_\frp(\ul M)
\]
of $A_\frp$-modules.
Hence $\wc V_\frp\ul{M}$ is crystalline if $\ul{M}$ has good reduction.
\end{example}

\subsection{The functors $\bH$ and $D_\crys$}
To study the category $\mathsf{Rep}_{Q_\frp}(\cG_L)^\crys$, we shall recall \textit{$z$-isocrystals with Hodge-Pink structures}.
This is the analogue of filtrated isocrystals in the theory of $p$-adic crystalline representations.

We fix a section $k\hra R$ of the reduction map $R\thra k$.
Then there exists an injective ring homomorphism $k\zp\hra L\zzb$ given by 
\[
z\mapsto \zeta+(\zz)\es\es\mathrm{and}\es\es \sum_{i}b_iz^i \mapsto \sum_{j=0}^\infty (\zz)^j\cdot \sum_{i}\binom{i}{j}b_i\zeta^{i-j}.
\]
Then $L\zzb$ and its fraction field $L\zzp$ become $k\zp$-vector spaces.

\begin{definition}
A \textit{$z$-isocrystal with Hodge-Pink structure over $R$} is a triple $\ul D=(D, \tau_D, \mathbf{q}_D^{})$ such that $(D,\tau_{D})$ is  a finite-dimensional $k\zp$-vector space together with a $k\zp$-isomorphism 
$\tau_D\colon \wh\sig\ast D\overset{\sim}{\lra} D$, and $\mathbf{q}_D^{}$ is an $L\zzb$-lattice in $D\otimes_{k\zp}L\zzp$ of full rank.
Then $\mathbf{q}_D^{}$ is called the \textit{Hodge-Pink lattice} of $\ul D$.
The dimension of $D$ is called the \textit{rank} of $\ul D$ denoted by $\rk \ul D$.
The integer $t_{\rm N}(\ul D):=\ord_z(\det \tau_D)$ is called the \textit{Newton slope of $\ul D$}.

For $\ul D=(D, \tau_D, \mathbf{q}_D^{})$ of rank $r$, write $\mathbf{p}_D^{}:=D\otimes_{k\zp}L\zzb$ for the tautological lattice in $D\otimes_{k\zp}L\zzp$.
Since $L\zzb$ is a principal ideal domain, we see by the elementary divisor theorem that 
\[
\mathbf{p}_D^{}=\bigoplus_{i=1}^rL\zzb\cdot v_i
\es \text{and} \es \mathbf{q}_D^{}=\bigoplus_{i=1}^{r}L\zzb \cdot (\zz)^{\mu_i}\cdot v_i
\]
 for a suitable basis $\{v_i\}$ of $\mathbf{p}_D^{}$ and integers $\mu_1\geq \mu_2\geq \cdots \geq \mu_r$. 
We call $\mu_1,\ldots,\mu_r$ the \textit{Hodge-Pink weights of $\ul D$} and put $t_{\rm H}(\ul D):=-(\mu_1+\cdots+\mu_r)$, which is called the \textit{Hodge slope of $\ul {D}$}. 

We say that $\ul D$ is \textit{weakly admissible} if $t_\textrm{H}(\ul D)=t_\textrm{N}(\ul D)$ and $t_\textrm{H}(\ul D')\leq t_\textrm{N}(\ul D')$ holds for any \textit{strictly subobject} $\ul D'$ of $\ul D$, that is, a $z$-isocrystal with Hodge-Pink structure $\ul D'$ satisfying that $D' \subset D$ is a $k\zp$-subspace with $\tau_D(\wh\sig\ast D')=D'$, $\tau_{D'}=\tau_{D}\mid_{\wh\sig\ast D}$, and $\mathbf{q}_{D'}^{}=\mathbf{q}_D^{}\cap D'\otimes_{k\zp}L\zzp$.

A \textit{morphism} $f\colon (D,\tau_{D}, \mathbf{q}_D^{})\ra (D', \tau_{D'}, \mathbf{q}_{D'}^{})$ of $z$-isocrystals with Hodge-Pink structures is a $k\zp$-linear homomorphism $f\colon D\ra D'$ such that $\tau_{D'}\circ \wh\sig\ast f=f\circ \tau_{D}$ and $(f\otimes \id)(\mathbf{q}_D^{}) \subset \mathbf{q}_{D'}^{}$.
\end{definition}

We denote by 
$
\mathsf{Crys}(R)
$
the category of $z$-isocrystals with Hodge-Pink structures over $R$.
Then \cite[Theorem 2.2.5]{Har11} says that the weakly admissible $z$-isocrystals with Hodge-Pink structures over $R$ form a full subcategory 
\[
\mathsf{Crys}(R)^\textrm{wa}
\]
of $\mathsf{Crys}(R)$ and the following properties hold.

\begin{itemize}
\item The category $
\mathsf{Crys}(R)^\textrm{wa}
$ is abelian, and  is closed under the formations of direct sums, tensor products, and duals.
\item A direct sum $\ul D\oplus \ul D'$ for $\ul D, \ul D'\in \mathsf{Crys}(R)$ is weakly admissible if and only if each summand is weakly admissible.
\end{itemize}

Since $R$ is discretely valued, it follows by \cite[Lemma 7.4]{GL11} that any local shtuka $\ul{\wh M}$ over $R$ is \textit{rigidified} in the sense of \cite[Definition 3.5.9]{HK20}; also see \cite[Lemma2.3.1]{Har11}, and 
there exists a category equivalence 
\[
\bH\colon \mathsf{QLS}(R)\lra \mathsf{Crys}(R)^\textrm{wa}
\] 
such that $\bH$ is a $k\zp$-linear exact tensor functor. 
It is called the \textit{mysterious functor}.
If $\bH(\ul{\wh M})=(D, \tau_D, \mathbf{q}_D^{})$, then we have
\[
(D,\tau_D) \cong (\wh M\otimes_{R\zb}k\zp, \tau_{\wh M}\otimes \id)
\]
and $\rk \bH(\ul{\wh M})=\rk\ul{\wh M}$; see \cite[\S 3 and \S 7]{GL11}, \cite[\S 2.3]{Har11}, and \cite[Example 3.5.7]{HK20} for details.
%

Combining the mysterious functor $\bH$ and the quasi-inverse of $\wc V_\frp\colon \mathsf{QLS}(R) \ra \mathsf{Rep}_{Q_\frp}(\cG_L)^\crys$, we obtain a category equivalence 
\[
D_\crys \colon  \mathsf{Rep}_{Q_\frp}(\cG_L)^\crys\lra \mathsf{Crys}(R)^\textrm{wa}.
\]
Hence $ \mathsf{Rep}_{Q_\frp}(\cG_L)^\crys$ is also an abelian category.
Moreover, we have the following.

\begin{proposition}\label{prop.closed}
The category $\mathsf{Rep}_{Q_\frp}(\cG_L)^\crys$ is closed under the formations of subobjects, quotients, direct sums, tensor products, and duals.
\end{proposition}

\begin{proof}
By  Lemma \ref{lemma.sub}, $\mathsf{Rep}_{Q_\frp}(\cG_L)^\crys$ is closed under subobjects.
Hence it is closed under quotients because it is an abelian category.
Remained properties come form $\mathsf{Crys}(R)^\textrm{wa}$ via the category equivalence $D_\crys$.
\end{proof}

%

\subsection{Weights of crystalline representations}
We define the \textit{weights} of $z$-adic crystalline representations as follows.
\begin{definition}\label{definition.wd}
Let $V\in \mathsf{Rep}_{Q_\frp}(\cG_L)^\crys$ be of dimension $r$ and put $D_\crys(V)=(D, \tau_D, \mathbf{q}_D^{})$.
Since $\wh\sig^{[k:\bF_\frp]}$ is identity on $k$, we have 
$\wh\sig^{[k:\bF_\frp]*}D \cong D $ and get a $k\zp$-linear automorphism 
\[
\tau_D^{[k:\bF_\frp]}:=\tau_D\circ \cdots \circ \wh\sig^{([k:\bF_\frp]-1)*}\tau_D \colon D 
\overset{\sim}{\lra} D.
\]
Define the \textit{characteristic polynomial of $V$} by 
\[
P(X; V)= \det(X-\tau_D^{[k:\bF_\frp]}), 
\]
which has coefficients in $Q_\frp$ by \cite[Lemma 8.1.4]{BP09}.
We say that $V$ has \textit{weights $w_1,\ldots, w_r$} if $P(X; V)$ has coefficients in $Q$ and decomposes as  
\[
P(X; V)=\prod_{i=1}^r (X-\alpha_i)\es\es \textrm{with}\es\es |\alpha_i|=(\# k)^{w_i}
\] 
for some $\alpha_i\in Q^\alg$.
\end{definition}

\begin{proposition}\label{proposition.weight}
If $\ul{M}$ is an $A$-motive of rank $r$ over $L$ with good reduction, then
 $P(X; \wc V_\frp \ul{M})=P(X; \ul{M})$.
In particular, $\ul{M}$ has weights $w_1,\ldots, w_r$ in the sense of Definition \ref{definition.wmot} if and only if $\wc V_\frp \ul{M}$ has weights $w_1,\ldots, w_r$ 
in the sense of Definition \ref{definition.wd}.
\end{proposition}

\begin{proof}

We view $R\zb$ as an $A_R$-algebra via $A_R\ra A_{\frp, R}/\fra_0\cong R\zb$ as in Example \ref{ex.motiveshtuka}.
Let $\ul{\cM}$ be a good model of $\ul{M}$ and  $\ul{\wh M}:=\ul{\wh M}_\frp(\ul{M})$ the associated local shtuka over $R$, so that $\ul{\wh M}=(\wh M, \tau_{\wh M})$ with $\wh M=\cM\otimes_{A_R}R\zb$ and  $\tau_{\wh M}=\tau_{\cM}^{\deg\frp}\otimes \id$.
For the reduction $\ul{M}_k=(M_k=\cM\otimes_{A_R}A_k, \tau_{k}=\tau_{\cM}\otimes \id)$ of $\ul{M}$, 
if we put $\bH(\ul{\wh M})=(D, \tau_D, \mathbf{q}_D)$, then we have 
\[
D\cong \wh{M}\otimes_{R\zb}k\zp\cong (\cM\otimes_{A_R}A_k)\otimes_{A_k}k\zp=M_k\otimes_{A_k}k\zp,
\]
where the $A_k$-algebra structure of $k\zp$ is $A_k\ra A_{\frp, k}/\fra_0\cong k\zb \hra k\zp$.
Since $\tau_D^{[k:\bF_\frp]}=\tau_k^{[k:\bF_q]}\otimes \id$, we have 
$
P(X; \wc V_\frp\ul{M})=P(X; \ul{M}).
$ 
\end{proof}

\begin{example}

 For each $n\in \bZ$, we define  the local shtuka $\ul{\mathbbm{1}}(n)$ to be  
 \[
 \ul{\mathbbm{1}}(n)=(R\zb\cdot \mathbf{e},\es  \wh\sig\ast \mathbf{e}\mapsto (z-\zeta)^{-n} \mathbf{e}),
 \]
 so that $\ul{\mathbbm{1}}(n)=\ul{\mathbbm{1}}(1)^{\ot n}$.
Thus the characteristic polynomial of $Q_\frp(n)=\wc V_\frp \ul{\mathbbm{1}}(n)$ is 
\[
P(X; Q_\frp(n))=X-z_{}^{-n[k:\bF_\frp]}\in Q[X].
 \]
The associated $z$-isocrystal with Hodge-Pink structure $\ul{D}=D_\crys(Q_\frp(n))=\bH(\ul{\mathbbm{1}}(n))$ is
\[
\ul{D}=\left(k\zp\cdot \mathbf{e},\es \wh\sig\ast \mathbf{e}\mapsto z^{-n}\mathbf{e},\es \mathbf{q}_D^{}=(\zz)^n\mathbf{p}_D^{}\right).
\] 
Hence
$
t_{\mathrm{N}}(Q_\frp(n))=t_{\mathrm{H}}(Q_\frp(n))=-n.
$
\end{example}

\begin{lemma}\label{lemma.intweight}
For each $n\in \bZ$, the weight of $Q_\frp(n)$ is 
\[
-n\left(1+\frac{1}{\deg\frp}\sum_{\frl\neq \frp}{\deg \frl}\ord_{\frl}(z)\right).
\]
In particular, the weight of $Q_\frp(n)$ is an integer under Condition \ref{condition.main}.
\end{lemma}

\begin{proof}
By $P(X; Q_\frp(n))=X-z_{}^{-n[k:\bF_\frp]}$, the weight of $Q_\frp(n)$ is $w=\log_{\#k}|z^{-n[k:\bF_\frp]}|$.
For each prime $\frl$ of $A$, the normalized $\frl$-adic absolute value $|\cdot|_{\frl}$  of $Q$ is given by 
\[
|a|_\frl=(\#\bF_\frl)^{-\ord_\frl(a)}=q_{}^{-\deg\frl \cdot \ord_\frl(a)}
\]
 for any $a\in Q$.
  It uniquely extends to $Q^\alg$, which is also denoted by $|\cdot|_\frl$.
By the product formula of absolute values (\cite[Chapter II, \S 12, Theorem]{CF67}), we have
\begin{align*}
w=\log_{\# k} |z^{-n[k:\bF_\frp]}|
&=\log_{\# k}\left( |z^{-n[k:\bF_\frp]}|_\frp\inv\cdot \prod_{\frl \neq \frp}|z^{-n[k:\bF_\frp]}|_\frl\inv\right)\\
&=\log_{\# k}(\# \bF_\frp)^{-n[k:\bF_\frp]}
+\sum_{\frl\neq \frp} \log_{\# k}(\#\bF_\frl)^{-n[k:\bF_\frp]\ord_{\frl}(z)}\\
&=\log_{\#k}(\#k)^{-n}+ \sum_{\frl\neq \frp} \log_{\#k}(\#k)^{-n\frac{\deg \frl}{\deg \frp}\ord_{\frl}(z)}\\
&=-n\left(1+\frac{1}{\deg \frp}\sum_{\frl\neq \frp}{\deg \frl}\ord_{\frl}(z)\right)
\end{align*}
as claimed.
\end{proof}

\begin{lemma}\label{lem.char}
Let $V\in \mathsf{Rep}_{Q_\frp}(\cG_L)^\crys$ be of dimension one. 
Then $\cG_L$ acts on $V$ by $\epsilon \cdot\chi_z^{-n}$, where $n=t_{\mathrm{N}}(V)$ and $\epsilon\colon\cG_L\ra Q_\frp^\times$ is a character trivial on ${\cI_L}$.

\end{lemma}

\begin{proof}
Let $\ul{\wh M}=(\wh M, \tau_{\wh M})$ be a local shtuka over $R$ of rank one such that $\wc V_\frp\ul{\wh M}=V$.
By \cite[Lemma 3.2.3]{HK20} there is a unit $u\in R\zb^\times$ such that $\tau_{\wh M}=u(z-\zeta)^n$.
Let $\ul{\wh N}=(R\zb\cdot \mathbf{e}, \tau_{\wh N})$ with $\tau_{\wh N}(\wh\sig\ast \mathbf{e})=u\mathbf{e}$ and denote by $\epsilon\colon \cG_L\to Q_\frp^\times$ the character coming form $\wc V_\frp\ul{\wh N}$. 
Then $\ul{\wh M}=\ul{\wh N}\otimes \ul{\mathbbm{1}}(-n)$ and so $\cG_L$ acts on $V$ via $\epsilon\cdot \chi_z^{-n}$.

By mimicking the argument of \cite[\S 5.2]{HS20}, we will show that 
$\epsilon$ is trivial on $\cI_L$. 
Since $\wc T_\frp\ul{\wh N}$ is of rank one, we may regard $\wc T_\frp\ul{\wh N}$ as an $A_\frp$-submodule of $L^\sep\zb$. 
Let $c=\sum_{i=0}^\infty c_iz^i\in \wc T_\frp\ul{\wh N}\subset L^\sep\zb$.
Then we have $\tau_{\wh N}(\wh \sig(c))=c$.
We set $u=\sum_{i=0}^\infty u_iz^i\in R\zb^\times$, so that $u_0\in R^\times$.
Since $\tau_{\wh N}=u$, we have $u\sum_{i=0}^\infty c_i^{\wh q}z^i=c=\sum_{i=0}^\infty c_iz^i$.
This implies
\[
c_0=u_0\cdot c_0^{\wh q} \es\es\es \text{and}\es\es\es c_i-u_0\cdot c_i^{\wh q}=\sum_{j=1}^{i}u_j\cdot c_{i-j}^{\wh q} \es\es (i\geq 1).
\]
In particular, we have 
\[
c_0^{\wh q-1}=u_0^{-1} \es\es\es \text{and}\es\es\es 
\frac{c_i}{c_0}-\left(\frac{c_i}{c_0}\right)^{\wh q}=\sum_{j=1}^i \frac{u_j}{u_0}\cdot \left(\frac{c_{i-j}}{c_0}\right)^{\wh q}
\]
and so for each $i\geq 0$, we have $c_i\in L^\ur$ the maximal unramified extension of $L$.
Since $\epsilon$ satisfies $\epsilon(g)c=\sum_{i=0}^\infty g(c_i)z^i$ for any $g\in \cG_L$, the restriction $\epsilon|_{\cI_L}$ is trivial.
\end{proof}

\section{Finiteness of torsion points}

In this section, for the $z$-adic cyclotomic extension $L_\infty=LF_{\frp,\infty}$ of $L$ and $A$-motives $\ul{M}$ over $L$ with good reduction, we study the finiteness of 
\[
(\wc V\ul{M}/\wc T\ul{M})^{\cG_{L_\infty}}=\prod_{\frl}(\wc V_\frl\ul{M}/\wc T_\frl\ul{M})^{\cG_{L_\infty}}.
\]
We notice that, in proving the above, we are allowed to replace $L$ by its finite extensions.

\subsection{Lubin-Tate characters}
 Let us recall Lubin-Tate theory over equal characteristic local fields, which gives an ingredient to 
prove a key lemma (Lemma \ref{lemma.key}).
We follow the expositions in \cite{HS20}. 

Let $E$ be a finite extension of $Q_\frp$ and denote by $\cO_E$ the valuation ring of $E$.
We write 
\[
\Gamma_E:=\Hom_{Q_\frp}(E,L^\alg)
\]
 for the set of $Q_\frp$-algebra homomorphisms $E\to L^\alg$, where the $Q_\frp$-algebra structure of $L^\alg$ is given by $\gamma\colon Q_\frp\to L\subset L^\alg$.
We assume that $\psi E:=\psi(E) \subset L$ for all $\psi\in \Gamma_E$.

Let $\bF$ be the residue field of $E$ and put $\wt q:=\#\bF$. 
Choosing a uniformizing parameter $y\in E$, we identify $\cO_E=\bF\dbl y \dbr$.
Take a $\psi\in \Gamma_E$.
Let $\wh G_\psi=\wh \bG_{a,R}=\Spf R\dbl X \dbr$ be the formal additive group over $R$ whose action of $\cO_E=\bF\dbl y \dbr$  is 
given by 
\begin{align*}
[b]&\colon X\mapsto \psi(b)\cdot X\es \text{for} \es b\in \bF,\\
[y]&\colon X\mapsto X^{\wt q}+\psi(y)\cdot X.
\end{align*}
Then $\wh G_\psi$ is a Lubin-Tate formal group over $R$ associated with $\cO_{\psi E}=\psi(\cO_E)$.

The $\cG_L$-action on the Tate module $T_\frp \wh G_\psi$ of $\wh G_\psi$ is computed as follows.
We take a system $\{\ell_n\}_{n=0}^\infty$ of elements of $L^\sep$ such that
$\ell_0^{\wt q-1}=-\psi(y)$ and $\ell_n^{\wt q}+\psi(y)\ell_n=\ell_{n-1}$. 
Namely, we have 
\[
[y](\ell_0)=\ell_0^{\wt q}+\psi(y)\ell_0=0\es \text{and}\es [y](\ell_n)=\ell_n^{\wt q}+\psi(y)\ell_n=\ell_{n-1}.
\]
Thus $\ell_{n-1}\in \wh G[y^n](L^\sep)$ and $T_\frp\wh G_\psi=\cO_E\cdot (\ell_{n-1})_{n}$.
For the power series 
\[
\ell_{y,\psi}^+:=\sum_{n=0}^\infty \ell_ny^n\in L^\sep\dbl y\dbr
\]
and the field $\psi E_y:=\psi E (\ell_n \colon n\in \bZ_{\geq 0}) \subset L^\sep$,
 we have the following.

\begin{definition}[{cf.\ \cite[p.\ 195]{HS20}}]
There is an isomorphism of topological groups
\[
\chi_{y,\psi}\colon \Gal(\psi E_y/\psi E) \overset{\sim}{\lra} \bF\dbl y \dbr^\times=\cO_{E}^\times
\]
satisfying $g\cdot\ell_{y,\psi}^+:=\sum_{n=0}^\infty g(\ell_n)y^n=\chi_{y,\psi}(g)\ell_{y,\psi}^+$ in $\psi E_y\dbl y\dbr$ for any 
$g\in \Gal(\psi E_y/\psi E)$.
The isomorphism $\chi_{y,\psi}$ is independent of the choice of $\ell_{y,\psi}^+$.
We call the character $\chi_{y,\psi}\colon \cG_{\psi E} \to E^\times$
the \textit{cyclotomic character of the field $E$}. 
We define the character 
\[
\chi_{\psi E}\colon \cG_{\psi E} \lra {\psi E^\times}
\]
by $\chi_{\psi E}(g):=\chi_{y,\psi}(g)|_{y= \psi(y)}$ for $g\in \cG_{\psi E}$, and call it 
the \textit{Lubin-Tate character associated with $\psi E$}.
\end{definition}

\begin{remark}\label{rem.LT}
(1) The characters $\chi_{y,\psi}$
and $\chi_{\psi E}$ depend on the choice of $y\in E$, but its restriction to the inertia subgroup $\cI_{\psi E}$ do not.

(2) If $E=Q_\frp, \psi=\gamma$, and $y=z$, then $\chi_{y,\psi}$ coincides with the 
$z$-adic cyclotomic character $\chi_z$.
\end{remark}

Denote by $U_{\psi E}:=\cO_{\psi E}^\times$ the unit group of $\psi E$.
We write $\psi E^\ur$ and $\psi E^\ab$ for the maximal unramified and  abelian extensions of $\psi E$, respectively.
We put $\cG_{\psi E}^\ab=\Gal(\psi E^\ab/\psi E)$ and $\cI_{\psi E}^\ab:=\Gal(\psi E^\ab/\psi E^\ur)$.
Let 
\[
(-, \psi E^\ab/\psi E)\colon \psi E^\times \lra \cG_{\psi E}^\ab
\]
 be the local Artin map of arithmetic normalization in local class field theory.
 It is known that this map induces an isomorphism
 \[
 (-, \psi E^\ab/\psi E)\colon U_{\psi E} \overset{\sim}{\lra} \cI_{\psi E}^\ab.
 \]
By\cite[p.\ 386, Corollary]{LT65}, for each $g\in \cI_{\psi E}$, we have 
\[
g|_{\psi E^\ab}=(\chi_{\psi E}(g)^{-1}, \psi E^\ab/\psi E).
\]
This means that if we regard $\chi_{\psi E}$ as a continuous character $U_{\psi E} \to \psi E^\times$ via the local Artin map, then we have $\chi_{\psi E}(x)=x^{-1}$ for each $x\in U_{\psi E}$.
Let $N_{L/\psi E}\colon L^\times \to \psi E^\times$ be the norm map.
Let $U_L$ and $\cI_{L}^\ab=\Gal(L^\ab/L^\ur)$ be similar as above.

\begin{lemma}\label{lem.com}
We have $\chi_{\psi E}(g)=N_{L/\psi E}(\chi_{L}(g))$ for each $g\in \cI_L$.
This makes the diagram
\[
\xymatrix@C=70pt@R=30pt{
U_L \ar[r]_-{\cong}^{\SSC (-, L^\ab/L)}\ar[d]_{N_{L/\psi E}}& \cI_{L}^\ab \ar[d]^{\mathrm{res}}\\
U_{\psi E} \ar[r]_-{\cong}^{\SSC (-, \psi E^\ab/\psi E)}& \cI_{\psi E}^\ab
}
\]
commutative, where $\mathrm{res}\colon \cI_{L}^\ab\to \cI_{\psi E}^\ab$ is the restriction map.
\end{lemma}

\begin{proof}
If $L/\psi E$ is separable, then it is well-known by local class field theory.
By the transitivity of the norm map, we may assume that  $L/\psi E$ is purely inseparable.
In this case, we have $[L:\psi E]=p^m$ for some $m \geq 0$ and so $N_{L/\psi E}(x)=x^{p^m}$.
We can take a purely inseparable extension $\wh E$ of $E$ such that 
$\psi\colon E\to L$ uniquely extends to an isomorphism $\wh \psi\colon \wh E \overset{\sim}{\lra} L$ of $Q_\frp$-algebras.
For any uniformizing parameter $\wh y\in \wh E$, its $p^m$-th power $\wh{y}^{p^m}\in E$ is a uniformizing parameter of $E$.
Since we consider $\chi_{L}$ and $\chi_{\psi E}$ on $\cI_L$, by Remark \ref{rem.LT} (1), we may assume that $y=\wh{y}^{p^m}$. 
For chosen solutions $\wh{\ell}_n \in L^\sep$ to the equations 
$\wh {\ell}_0^{\wt q-1}=-\wh\psi(\wh y)$ and $\wh{\ell}_n^{\wt q}+\wh\psi(\wh y)\wh{\ell}_n=\wh{\ell}_{n-1}$, 
we put  $\ell_{\wh y,\wh\psi}^+:=\sum_{n=0}^\infty \wh{\ell}_n\wh y^n\in L^\sep\dbl \wh y \dbr$.
Then the cyclotomic character $\chi_{\wh y,\wh\psi}$ of $\wh E$ satisfies 
$ 
g\cdot\ell_{\wh y,\wh\psi}^+=\chi_{\wh y,\wh\psi}(g) \wh{\ell}_{\wh y,\wh\psi}^+
$
for each $g\in \cI_L$.
Putting $\ell_n:=(\wh{\ell}_n)^{p^m}$ for each $n$, we have
$\ell_0^{\wt q-1}=-\psi(y)$, $\ell_n^{\wt q}+\psi(y)\ell_n=\ell_{n-1}$, and 
$\ell_{y,\psi}^+=\sum_{n=0}^\infty \ell_ny^n=(\ell_{\wh y, \wh\psi}^+)^{p^m}$.
Since 
\[
\chi_{y,\psi}(g) \ell_{y,\psi}^+=g\cdot\ell_{y,\psi}^+=(g\cdot\ell_{\wh y,\wh\psi}^+)^{p^m}=\chi_{\wh y,\wh\psi}(g)^{p^m}\ell_{y,\psi}^+,
\]
we have 
$
\chi_{y,\psi}(g)=\chi_{\wh y,\wh\psi}(g)^{p^m}.
$
By $\psi(y)=\wh\psi(\wh y)^{p^m}$ and $\wh\psi\wh E=L$, we get 
\[
\chi_{\psi E}(g)=\chi_{L}(g)^{p^m}=N_{L/\psi E}(\chi_{L}(g)).
\]
This implies the commutativity of the diagram.
\end{proof}

The next lemma is crucial for our results (cf.\ \cite[Lemma 3.5]{KT13}).

\begin{lemma}\label{lemma.key}
Let $V$ be a two-dimensional irreducible crystalline representation of $\cG_L$ on which $\cG_{L_\infty}$ acts trivially.
Then there is a finite separable extension $L'$ of $L$ such that $V\mid_{\cG_{L'}}\cong Q_\frp(n)^{\op 2}$ 
for some $n\in \bZ$.  
\end{lemma}

\begin{proof}
Let $\ul{\wh M}$ be a local shtuka over $R$ with $\wc V_\frp\ul{\wh M}=V$, so that $\rk \ul{\wh M}=2$.
Let  
$\rho\colon \cG_L\to \Aut_{Q_p}(\wc V_\frp\ul{\wh M})$ be  a continuous homomorphism describing the $\cG_L$-action on $\wc V_\frp\ul{\wh M}$.
Schur's lemma implies that $\End_{Q_\frp[\cG_L]}(\wc V_\frp\ul{\wh M})$ is a division algebra over $Q_\frp$.
Let $E$ be the maximal (commutative) subfield of $\End_{Q_\frp[\cG_L]}(\wc V_\frp\ul{\wh M})$.
Since $\rho$ factors through the abelian group $\Gal(L_\infty/L)$,  
we see that $\wc V_\frp\ul{\wh M}$ becomes an $E$-vector space and $\rho$ is regarded as a character $\rho\colon \cG_L\to E^\times$.
In particular, $E$ is a quadratic extension of $Q_\frp$. 
Since $E/Q_\frp$ is Galois by $p\neq 2$, the image $\psi E\subset L^\alg$ of $\psi\in \Gamma_E$ is independent of $\psi$.
We set $\wt E:=\psi E$, which is Galois over $F_\frp$.
Replacing $L$ by a finite separable extension, we may assume that 
$L/F_\frp$ is a normal extension and $\wt E\subset L$.

Since the functor $V_\frp\colon \mathsf{QLS}(R)\to \mathsf{Rep}_{Q_\frp}(\cG_L)^\crys$ is a category equivalence, 
the $E$-action on $\wc V_\frp\ul{\wh M}$ induces an inclusion $E\subset \QEnd_{R}(\ul{\wh M})$.
By \cite[Proposition 4.2]{HS20}, we may assume that $\cO_E\subset \End_R(\ul{\wh M})$.
Hence $\ul{\wh M}$ has \textit{complex multiplication by $\cO_E$} in the sense of \cite[Definition 4.3]{HS20}.
For the \textit{local CM-type} $\{d_\psi\in \bZ \colon \psi \in \Gamma_E\}$ of $\ul{\wh M}$ ; see \cite[Definition 4.11]{HS20}, 
we define $\wt \rho\colon \cG_{\wt E}\to E^\times$ by 
\[
\wt \rho=\prod_{\psi\in\Gamma_E}\psi^{-1}\circ \chi_{\wt E}^{-d_\psi}.
\]
It follows by \cite[Proposition 5.11]{HS20} that $\rho|_{\cI_L}=\wt \rho|_{\cI_L}$.
Now $\rho$ factors through $\Gal(L_\infty/L)$ and the image of the restriction map $\cI_L\to \Gal(L_\infty/L)$ is of finite index in $\Gal(L_\infty/L)$.  
Replacing $L$ by its suitable finite separable extension, 
we have $\rho=\wt \rho|_{\cG_L}$. 
Then $\wt \rho$ is trivial on $\cG_{\wt E}\cap \cG_{L_\infty}=\cG_{\wt E_\infty}$ and hence it factors through $\Gal(\wt E_\infty/\wt E)$.
By the fact that the restriction map $\Gal(\wt E_\infty/\wt E)\to \Gal(F_{\frp,\infty}/F_\frp)$ is injective, 
 $\Gal(\wt E_\infty/\wt E)$ is regarded as a subgroup of $\Gal(F_{\frp,\infty}/F_\frp)$ whose index is two.
Thus one can take an element $\alpha\in \Gal(F_{\frp,\infty}/F_\frp)$ such that $\Gal(F_{\frp,\infty}/F_\frp)=\Gal(\wt E_\infty/\wt E)\amalg \alpha\Gal(\wt E_\infty/\wt E)$ and $\alpha^2\in \Gal(\wt E_\infty/\wt E)$.
We choose an element $\lambda\in \wt E^\sep$ such that $\lambda^2=\wt \rho(\alpha^2)$.
Define $\wt\rho_{F_\frp}\colon \Gal(F_{\frp,\infty}/F_\frp)\to \wt E(\lambda)^\times$  by $\wt\rho_{F_\frp}(\alpha\beta)=\lam\wt\rho(\beta)$ and $\wt\rho_{F_\frp}(\beta)=\wt\rho(\beta)$ for each $\beta\in \Gal(\wt E_\infty/\wt E)$.
It is a group homomorphism and extends to $\wt \rho_{F_\frp}\colon \cI_{F_\frp}\to \wt E(\lam)^\times$ by the surjective map $\cI_{F_\frp}\to \Gal(F_{\frp,\infty}/F_\frp) ; g\mapsto g|_{F_{\frp,\infty}}$.
By construction, $\wt \rho(g)=\wt\rho_{F_\frp}(g)$ for each $g\in \cI_{\wt E}$.
Since $\rho|_{\cI_L}$, $\wt \rho|_{\cI_{\wt E}}$, and $\wt\rho_{F_\frp}$ factor through $\cI_L^\ab, \cI_{\wt E}^\ab$, and $\cI_{F_\frp}^\ab$, 
 we see by Lemma \ref{lem.com} that the diagram
 \[
\xymatrix@C=70pt@R=15pt{
U_L \ar[r]_-{\cong}^{\SSC (-, L^\ab/L)}\ar[d]_{N_{L/{\wt E}}}& \cI_{L}^\ab \ar[d]^{\mathrm{res}}\ar[r]^{\rho} & E^\times \ar@{=}[d] \\
U_{{\wt E}} \ar[r]_-{\cong}^{\SSC (-, {\wt E}^\ab/{\wt E})} \ar[d]_{N_{{\wt E}/F_\frp}}& \cI_{{\wt E}}^\ab\ar[d]^{\mathrm{res}} \ar[r]^{\wt \rho}& E^\times \ar@{^{(}-_>}[d]\\
U_{F_\frp}\ar[r]_-{\cong}^{\SSC (-, {F_\frp^\ab}/F_\frp)} & \cI_{F_\frp}^\ab \ar[r]^{\wt\rho_{F_\frp}}& E(\lam)^\times
}
\]
is commutative.
Regarding $\rho|_{\cI_L}$ as a character $U_L\to E^\times$ via the local Artin map, we obtain
\begin{equation}\label{eq.G1}
\rho(h(x))=\rho(x)\es\es\text{for each}\ h\in \Aut_{F_\frp}(L)\ \text{and}\ x\in U_L
\end{equation}
because $\rho|_{\cI_L}$ factors through the norm map $N_{L/F}\colon U_L\to U_{F_\frp}$.

Take an element $x\in U_L$.
We have assumed that $L/F_\frp$ is normal, so that the separable closure $L_0$ of $F_\frp$ in $L$ is Galois over $F_\frp$ and $\wt E\subset L_0$.
Since $L$ is purely inseparable over $L_0$,  we are allowed to identity $\Aut_{F_\frp}(L)=\Gal(L_0/F_\frp)$ via $h\mapsto h|_{L_0}$.
We fix a $\phi\in\Gamma_E$ and put $\Gal(\wt E/F_\frp)=\{\id, s\}$.
Then $\Gamma_E=\{\phi, s\circ\phi\}$.
Choosing an $\wh s\in \Gal(L_0/F_\frp)$ such that $\wh s|_{\wt E}=s$, we obtain 
$\Gal(L_0/F_\frp)=\Gal(L_0/\wt E)\amalg \wh s^{-1}\Gal(L_0/\wt E)$.
Put $x_0:=N_{L/L_0}(x)=x^{p^m} \in U_{L_0}$, where $p^m=[L:L_0]$.
Then for $N(x):=\prod_{\psi\in\Gamma_E}\psi^{-1}\circ N_{L/\wt E}(x)\in E^\times$, we have 
\begin{align*}
N(x)&=(\phi^{-1}\circ N_{L_0/\wt E}(x_0)) \cdot (\phi^{-1}\circ {s}^{-1}\circ N_{L_0/\wt E}(x_0))\\
&=\phi^{-1}\left(
\prod_{g\in \Gal(L_0/\wt E)}g(x_0)\cdot \prod_{g\in \Gal(L_0/\wt E)} \wh{s}^{-1}\circ g(x_0)
\right)\\
&=\phi^{-1}\left(
\prod_{h\in \Gal(L_0/F_\frp)}h(x_0)
\right)\\
&=\phi^{-1}\circ N_{L/F_\frp}(x).
\end{align*}
This implies that $N(x)\in Q_\frp^\times$ and 
\begin{equation}\label{eq.G2}
N(h(x))=N(x) \es \text{for each}\es h\in \Aut_{F_\frp}(L).
\end{equation}
Setting $d:=d_{\phi}-d_{s\circ\phi}$, for the character $\rho\colon U_L\to E^\times$, we have
\begin{align*}
\rho(x)=\prod_{\psi\in \Gamma_E}\psi^{-1}\circ N_{L/\wt E}(x^{-1})^{-d_\psi}&=\prod_{\psi\in\Gamma_E}\psi^{-1}\circ N_{L_0/\wt E}(x_0)^{d_\psi}\\
&=(\phi^{-1}\circ N_{L_0/\wt E}(x_0)^{d_\phi})\cdot (\phi^{-1}\circ s^{-1}\circ N_{L_0/\wt E}(x_0)^{d_{s\circ \phi}})\\
&=(\phi^{-1}\circ N_{L_0/\wt E}(x_0)^d)\cdot N(x)^{d_{s\circ \phi}}.
\end{align*} 
For each $h\in \Gal(L_0/F_\frp)$, we have $N_{L_0/\wt E}(x_0)^d=N_{L_0/\wt E}(h(x_0))^d$
by (\ref{eq.G1}) and  (\ref{eq.G2}).
Since $\Gal(L_0/\wt E)$ is a normal subgroup of $\Gal(L_0/F_\frp)$, 
we also have  
\[
N_{L_0/\wt E}(x_0)^d=h\circ N_{L_0/\wt E}(x_0)^d.
\]
This implies that $N_{L_0/\wt E}(x_0)^d\in F_\frp^\times$ and so $\rho(x)\in Q_\frp^\times$.

Consequently, the character $\rho\colon \cG_L\to E^\times$ values in $Q^\times_\frp$ on the image of the restriction map $\cI_L\to \Gal(L_\infty/L)$, which is of finite index in $\Gal(L_\infty/L)$.
Replacing $L$ by a finite separable extension, we have
$\rho(\cG_L)\in Q_\frp^\times$.
This means that  $\rho=\omega^{\op 2}$ for some crystalline character $\omega\colon \cG_L\to Q^\times_\frp$.
By Lemma \ref{lem.char}, we have $\omega=\epsilon\cdot \chi_z^n$ for some $n\in \bZ$ and $\epsilon\colon\cG_L\to Q_\frp^\times$ satisfying $\epsilon(\cI_L)=1$.
Hence 
we have $\omega|_{\cG_{L'}}=\chi_z^n$ for some finite separable extension $L'$ of $L$.
The lemma is proved.
\end{proof}

\subsection{Proof of results}

Let $\frl$ be a prime of $A$. 
We begin with the following abstract lemma. 

\begin{lemma}[{cf.\ \cite[Lemma 2.1]{KT13}}]\label{lemma.fix}
Let $V$ be a finite-dimensional $Q_\frl$-representation of a group $\cG$ and $T$ a $\cG$-stable $A_\frl$-lattice in $V$.
Then the following assertions are equivalent.
\begin{itemize}
\item[(1)] $V^\cG\neq 0$. 
\item[(2)] $(V/T)^\cG$ is infinite.
\end{itemize}
\end{lemma}

\begin{proof}
Let $y\in A_\frl$ be a uniformizing parameter.
If there is a  non-zero $v\in V^{\cG}$, then it gives an infinite subset $\{y^{-n}v+T : n\in \bZ_{>0}\}$ of $(V/T)^{\cG}$.
Conversely, we 
suppose that $(V/T)^{\cG}$ is infinite. 
Since $V=\bigcup_{n=0}^\infty y^{-n}T$, we have $(V/T)^{\cG}=\bigcup_{n=0}^\infty \left(y^{-n}T/T\right)^{\cG}$.
It follows that $(y^{-(n-1)}T/T)^{\cG}\subsetneq (y^{-n}T/T)^{\cG}$ for any $n\geq 1$, so that $(yT/y^{n}T)^{\cG}\subsetneq (T/y^nT)^{\cG}$.
Since all $(T/y^nT)^{\cG}\setminus (yT/y^nT)^{\cG}$ are non-empty finite sets, we see that 
\[
T^{\cG}\setminus yT^{\cG}=\varprojlim_n (T/y^nT)^{\cG}\setminus (yT/y^nT)^{\cG}
\] is also non-empty, so $V^{\cG}\neq 0$.
\end{proof}

For all $\frl$ distinct from $\frp$, we have the following.

\begin{proposition}\label{proposition.frl}
Let $\ul{M}$ be an $A$-motive over $L$ with good reduction.
Let $L'$ be a Galois extension of $L$ such that $L'\cap L^\ur$ is finite over $L$.
If $\ul{M}$ has only non-zero weights, then $\prod_{\frl\neq \frp}(\wc V_\frl\ul{M}/\wc T_\frl\ul{M})^{\cG_{L'}}$ is finite.
\end{proposition}

\begin{proof} 
Extending $L$ by a finite extension,
we may assume that $L'/L$ is  totally ramified and so the restriction map $\cI_L\ra \Gal({L'}/L)$ is surjective.
Since $\wc V_\frl\ul{M}$ is unramified for all $\frl\neq \frp$, the $\cG_{L}$-action on $(\wc V_\frl\ul{M})^{\cG_{L'}}$ is trivial.
This means that, if $(\wc V_\frl\ul{M})^{\cG_{L'}}\neq 0$, then $1\in Q$ is an eigenvalue of $\rho_{\ul{M},\frl}(\Frob_k)$; however, this 
contradicts that all weights of $\ul{M}$ are non-zero.
Thus each $(\wc V_\frl\ul{M}/\wc T_\frl\ul{M})^{\cG_{L'}}$ is finite by Lemma \ref{lemma.fix}.
If we show $(\wc V_\frl\ul{M}/\wc T_\frl\ul{M})^{\cG_{L'}}=0$ for almost all $\frl$, we get  the conclusion.
To prove this, it is enough to check that $(\wc T_\frl\ul{M}/\frl \wc T_\frl\ul{M})^{\cG_{L'}}=0$.
We denote by $\bar{\rho}_{\ul{M},\frl}\colon \cG_L \ra \Aut_{\bF_\frl} \left(\wc T_\frl\ul{M}/\frl \wc T_\frl\ul{M}\right)$ the corresponding $\bF_\frl$-representation.
It follows by Lemma \ref{lemma.Frob} that 
\[
P(X; \ul{M}) \equiv \det(X-\bar{\rho}_{\ul{M},\frl}(\Frob_k)) \pmod \frl
\]
for any $\frl\neq \frp$.
By $P(X; \ul{M}) \in Q[X]\cap A_\frl[X]$, we have $P(1;\ul{M})\in Q$ and $\ord_\frl P(1;\ul{M}) \geq 0$.
Since $\ul{M}$ has only non-zero weights,  we see that $P(1; \ul{M})\neq 0$. 
Thus for all but finitely many primes $\frl$, we have $P(1; \ul{M}) \not\equiv 0\pmod \frl$.
This implies that for almost all $\frl$, 
all eigenvalues of $\bar{\rho}_{\ul{M},\frl}(\Frob_k)$ are distinct from $1$, so that $(\wc T_\frl\ul{M}/\frl \wc T_\frl\ul{M})^{\cG_{L'}}=0$.
\end{proof}

The following is our main result.

\begin{theorem}\label{theorem.main}
Let $\ul{M}$ be an
$A$-motive over $L$ of rank two with good reduction and 
$z\in Q$ a uniformizing parameter at $\frp$ satisfying  Condition \ref{condition.main}.
If $p\neq 2$ and $\ul{M}$ has no integral weights, then $(\wc V\ul{M}/\wc T\ul{M})^{\cG_{L_\infty}}$ is finite.
\end{theorem}

\begin{proof}
By Proposition \ref{proposition.frl}, it remains to check that $(\wc V_\frp\ul{M})^{\cG_{L_\infty}}=0$. 
 We assume that $(\wc V_\frp\ul{M})^{\cG_{L_\infty}}\neq 0$ 
and take a non-zero irreducible $Q_\frp$-subrepresentation $V \subset (\wc V_\frp\ul{\wh M})^{\cG_{L_\infty}}$, which is crystalline by Proposition \ref{prop.closed}.
Since $V$ is of dimension $\leq 2$, Lemmas \ref{lem.char} and \ref{lemma.key} imply that, after extending $L$ by a finite extension, 
there is an integer $n\in \bZ$ such that $Q_\frp(n)\subset V$.
Consequently, $P(X; \ul{M})$ is divisible by $P(X; Q_\frp(n))$ and hence at least one of the weights of $\ul{M}$ is an integer by Lemma \ref{lemma.intweight}.
But this is a contradiction.
\end{proof}

\begin{corollary}\label{corollary.Amod}
Let $\ul{G}$ be an abelian Anderson $A$-module of rank two over $L$ with good reduction and $z\in Q$ a uniformizing parameter at $\frp$ satisfying  Condition \ref{condition.main}.
If $p\neq 2$ and $\ul{M}(\ul{G})$ has no integral weights, then 
$\ul{G}(L_\infty)_{\tor}$ is finite.
In particular, if $p\neq 2$ and $\ul{G}$ is a Drinfeld $A$-module of rank two with good reduction, then $\ul G(L_\infty)_\tor$ is finite. 
\end{corollary}

\begin{proof}
We put $\ul{M}:=\ul M(\ul G)^\vee$. 
Since $\wc V_\frl\ul M \cong \Hom_{Q_\frl}(\wc V_\frl \ul{M}(G), Q_\frl)\cong V_\frl \ul{G}$, we see that $\ul{M}$ also has good reduction and  no integral weights.
For each prime $\frl$ of $A$, we have a $\cG_L$-equivariant isomorphism 
\[
(\wc V_\frl\ul{M}/\wc T_\frl\ul{M})^{\cG_{L_\infty}} \cong 
(V_\frl\ul{G}/T_\frl \ul G)^{\cG_{L_\infty}} \overset{\sim}{\lra} \ul{G}[\frl^\infty](L_\infty):=\bigcup_{n=1}^\infty \ul{G}[\frl^n](L_\infty).
\]
It follows by Theorem \ref{theorem.main} that $\ul{G}[\frl^\infty](L_\infty)$ is finite for all $\frl$, and that  $\ul{G}[\frl^\infty](L_\infty)=0$ for almost all $\frl$.
Hence $\ul{G}(L_\infty)_\tor$ is finite.
\end{proof}

\subsection{Examples of Drinfeld modules with infinite $L_\infty$-valued torsion}

Assume that $A=\bF_q[t]$ and $z\in Q$ is the monic generator of $\frp$.
In this case,  Condition \ref{condition.main} holds.
For the $\gamma\colon A\to L$, we put $\theta:=\gamma(t)\in L$.
Then we have $A_L\cong L[t]$ and $J_L=(t-\theta)$.

The \textit{Carlitz module} $\ul{C}=(\bG_{a,L}, \Phi)$ over $L$ is a Drinfeld $A$-module of rank one determined by  
\[
\Phi_t=\theta+\tau\in L\{\tau\}.
\]
It plays the role of the multiplicative group $\bG_m$ and gives a function field analogue of cyclotomic theory. 
The associated $A$-motive is given by 
\[
\ul{M}(\ul C)=(L[t]\cdot \mathbf{e},\es \sig\ast\mathbf{e}\mapsto (t-\theta)\mathbf{e})
\]
called the \textit{Carlitz motive} over $L$.
We see that $\ul{C}$ and $\ul{M}(\ul C)$ have good reduction.
It follows by \cite[Example 3.2.7]{HK20} that the local shtuka associated with $\ul{M}(\ul C)$ is  
$
 \ul{\mathbbm{1}}(-1)
$
 and so we have $\wc V_\frp\ul{M}(\ul C)=Q_\frp(-1)$ and $ V_\frp\ul{C}= Q_\frp(1)$.
 
\begin{example}\label{ex.C1}
For two integers $m, n\in \bZ$, the rank-two $A$-motive 
\[
\ul{M}_{m,n}:=\ul{M}(\ul C)^{\ot m} \oplus \ul{M}(\ul C)^{\ot n}
\]
 has good reduction and satisfies 
$\wc V_\frp\ul{M}_{m,n}\cong Q_\frp(-m)\oplus Q_\frp(-n)$. Hence it has weights $m, n$ by Lemma \ref{lemma.intweight}.
Clearly we see that $(\wc V_\frp\ul{M}_{m,n}/\wc T_\frp\ul{M}_{m,n})^{\cG_{L_\infty}}$ is infinite.
Hence if one removes the assumption having non-integral wights, then Theorem \ref{theorem.main} dose not hold.
\end{example}

We next consider the case where $A$-motives have rank equal to the characteristic $p$.
For simplicity, we assume that $q=p$ and so $A=\bF_p[t]$.
 Take a unique element 
$\wt t\in Q^\alg$ such that $\wt t^p=t$.
We set $\wt A:=\bF_p[\wt t]$ and assume that $L$ contains the unique 
$p$-th root $\wt \theta$ of $\theta=\gamma(t)$.
Then $\gamma$ is uniquely extended to $\gamma\colon \wt A\to L$ 
by $\gamma(\wt t)=\wt \theta$.
Putting $\mu^{(p)}:=\sum_{i=0}^n \lam^p\tau^i$ for each $\mu=\sum_{i=0}^n\lam_i\tau^i \in L\{\tau\}$,  we get an $\bF_p$-algebra homomorphism $L\{\tau\}\to L\{\tau\} ; \mu\mapsto \mu^{(p)}$.
 
\begin{definition}
Define $\ul{\wt C}=(\bG_{a,L}, \phi)$ to be a Drinfeld $\wt A$-module such as  
$\phi_{\wt t}=\wt \theta+\tau\in L\{\tau\}$. 
This is the ``Carlitz module for $\wt A$''.
\end{definition}

For each $a\in A$, let $\wt a\in \wt A$ be the unique element 
satisfying $\wt a^p=a$. 
Then $A\to \wt A; a\mapsto \wt a$ is an isomorphism of $\bF_p$-algebras.

\begin{lemma}\label{lemma.two}
The map 
\begin{equation}\label{eq.map}
\ul{\wt C}[\wt{a}](L^\sep) \to \ul C[a](L^\sep); x\mapsto x^p
\end{equation}
is an isomorphism as abelian groups. 
\end{lemma}

\begin{proof}
We have $\phi_{\wt a}^{(p)}=\Phi_a$ for each $\wt{a}\in \wt A$, where $\ul C=(\bG_{a,L}, \Phi)$.
For each $x\in \ul{\wt{C}}[\wt{a}](L^\sep)$, we have 
\[
\Phi_a(x^p)=\phi_{\wt{a}}^{(p)}(x^p)=\phi_{\wt{a}}(x)^p=0.
\]
Thus $x^p\in \ul{C}[a](L^\sep)$ and so (\ref{eq.map}) is well-defined and injective.
Since $\ul{\wt{C}}[\wt{a}](L^\sep)$ and $\ul C[a](L^\sep)$ have the same cardinality (because $\wt{A}/(\wt{a})\cong A/(a)$), the map (\ref{eq.map}) is bijective.
\end{proof}

\begin{example}\label{ex.C2}
We define $\ul{\wt G}:=(\bG_{a,L}, \phi|_{A})$, which is a Drinfeld $A$-module such that 
\[
\phi_t=(\wt \theta+\tau)^p=\theta+\cdots+\tau^p.
\]
Thus $\ul{\wt G}$ is of rank $p$ and has good reduction.
Recall that $\frp=(z)$.
For each $n\geq 1$, we have 
$
\ul{\wt G}[z^n](L^\sep)=\ul{\wt C}[z^n](L^\sep)
$ as abelian groups by definition.
Take an $x\in \ul{\wt G}[z^n](L^\sep)$.
Then $x^p\in \ul{C}[z^{pn}](L^\sep)$ by Lemma \ref{lemma.two} 
and so $g(x^p)=x^p$ for each $g\in \cG_{L_\infty}$, which implies $g(x)=x$.
Hence $\ul{\wt G}[z^n](L_\infty)=\ul{\wt G}[z^n](L^\sep)$.
This means that $\ul{\wt G}(L_\infty)_\tor$ is infinite.
Therefore, the assumption $p\neq 2$ is needed for Theorem \ref{theorem.main}.
\end{example}

\section{Application to  Drinfeld modules over global function fields}

In this last section, we suppose that $Q=\bF_q(t)$ and the distinguished place $\infty$ corresponds to $1/t$.
Then $A=\bF_q[t]$.
Let $F:=\bF_q(\theta)$
be the rational function field in variable $\theta$ 
 equipped with the $\bF_q$-algebra homomorphism $\gamma\colon A\to F$ such that $\gamma(t)=\theta$.
Take separable and algebraic closures of $F$ with $F\subset F^\sep\subset F^\alg$.
Let $K$ be a finite extension of $F$ in $F^\alg$.
We notice that, in what follows, the symbol $\frp$ is used for an arbitrarily prime of $A$.

Now the Carlitz module $\ul{C}=(\bG_{a,F},\Phi)$ over $F$ is determined by $\Phi_t=\theta+\tau\in F\{\tau\}$.
For a prime  $\frp$ of $A$ and $n\in \bZ_{>0}$, 
we set 
\[
F(\frp^n):=F(\ul{C}[\frp^n](F^\sep)).
\]
It is Galois over $F$ with an isomorphism
\[
\Gal(F(\frp^n)/F) \overset{\sim}{\lra} (A/\frp^n)^\times
\]
induced by the Galois action on $\ul C[\frp^n](F^\sep)$.
It is known that  $F(\frp^n)/F$ is unramified outside $\frp\infty$, and that $\frp$ is totally ramified and $\infty$ is tamely ramified
 with ramification index $q-1$.
We denote by 
\[
\chi_{\frp}^{} \colon \cG_F \ra A_\frp^\times
\]
the character coming from the Galois action on $T_\frp\ul C$.
If we set $F(\frp^\infty):=\bigcup_{n=1}^\infty F(\frp^n)$, then 
the character $\chi_{\frp}$  factors through $\Gal(F(\frp^\infty)/F)$ and  induces an isomorphism
\[
\chi_{\frp}\colon 
\Gal(F(\frp^\infty)/F) \overset{\sim}{\lra} A_\frp^\times
\]
as topological groups.
We set 
\[
\wt F:=F(\ul{C}(F^\sep)_\tor), 
\]
which coincides with the composite of all  $F(\frp^\infty)$. 
By the Kronecker-Weber-type theorem for $F=\bF_q(\theta)$ due to Hayes \cite[Theorem 7.1]{Hay74}, the maximal abelian extension $F^\ab$ of $F$ 
is given as the composite of three linearly disjoint fields 
$\bF_q^\alg$, $\wt F$, and $\cF$.
Here, $\bF_q^\alg$ is the algebraic closure of $\bF_q$ in $F^\sep$ and $\cF$ is an extension of $F$ such that the place associated with $1/\theta$ is totally wildly ramified.

Putting  $K(\frp^n):=KF(\frp^n)$, $K(\frp^\infty):=KF(\frp^\infty)$,  and  
$
\wt K:=K\wt F
$,  we get an analogue of Ribet's theorem in \cite[Appendix]{KL81}.

\begin{theorem}\label{theorem.Ribet}
Let $\ul G$ be a Drinfeld $A$-module of rank two over $K$.
If $p\neq 2$, then $\ul G(\wt K)_\tor$ is finite. 
\end{theorem}

This follows from the next two theorems.

\begin{theorem}\label{theorem.R1}
Let $\ul G$ be a Drinfeld $A$-module of rank $r\geq 2$ over $K$.
Then $\ul G[\frp^\infty](\wt K)=0$ for almost all $\frp$.
\end{theorem}

\begin{theorem}\label{theorem.R2}
Let $\ul G$ be a Drinfeld $A$-module of rank two over $K$.
If $p\neq 2$, then $\ul G[\frp^\infty](\wt K)$ is finite for all $\frp$.
\end{theorem}

To prove them,  we shall prepare some lemmas.
Let  $z_\frp\in A$ be the monic generator of $\frp$ 
and set $\zeta_\frp:=\gamma(z_\frp)\in F$.
Then $\gamma$ induces an isomorphism  
\[
\gamma\colon Q_\frp \overset{\sim}{\ra} F_\frp:=\bF_\frp\dpl \zeta_\frp \dpr.
\]
We note that Condition \ref{condition.main} holds for $z_\frp$.
As in \S 1, we consider the $z_\frp$-adic cyclotomic extension 
$
F_{\frp,\infty}
$
 of $F_\frp$.
Taking an embedding $F^\alg \hra F_\frp^\alg$, we may view $\cG_{F_\frp}$ 
as a subgroup of $\cG_F$.
Since the restriction of $\chi_{\frp}$ to $\cG_{F_\frp}$
coincides with the $z_\frp$-adic cyclotomic character 
$\chi_{z_\frp}$, we have $F(\frp^\infty) \subset F_{\frp,\infty}$.

Any place of $F$ not lying above $1/\theta$, called a \textit{finite place}, is corresponds to a prime of $A$ via $\gamma\colon A\to F$.
We denote by $\Sigma_F$ the set consisting of all finite places of $F$, which is identified with the set of all primes of $A$.
For each place $v$ of $K$, we write $K_v$ for the completion of $K$ at $v$.
If $v$ divides some $\frp\in \Sigma_F$, then it is called a \textit{finite place} of $K$, and then $\gamma\colon A\to F\subset F_\frp\subset K_v$ factors through the valuation ring of $K_v$.
Denote by $\Sigma_K$ the set consisting of all finite places of $K$.

\begin{lemma}\label{lemma.unip}
For each Drinfeld $A$-module $\ul G$ over $K$, there is a finite separable extension $K'/K$ such that 
for each finite place $w$ of $K'$ and each $\frp\in \Sigma_F$ with $w\nmid \frp$, the $\cI_{K_w'}$-action on $T_\frp \ul G$ is unipotent. 
\end{lemma}

\begin{proof}
Since $\ul G$ has good reduction at almost all finite places of $K$, we may consider only remained $v_1,\ldots,v_m\in \Sigma_K$ at which $\ul G$ is not good.
By \cite[Lemma 2.10]{DD99}, we may assume that $\ul G$ has \textit{stable reduction} (see \cite[Definition 4.10.1]{Gos96}) at all $v_i$.
For each $\frp\in \Sigma_F$ with $v_i\nmid \frp$ for some $i$,  
 the Tate uniformization (see \cite[\S 7]{Dri74}) implies that there are a Drinfeld $A$-module $\ul G'$ of rank $\leq \rk \ul G$ over $K_{v_i}$ with good reduction and a $\cG_{K_{v_i}}$-stable finite generated $A$-submodule $\Lambda_{v_i}$ of $K_{v_i}^\sep$  
such that they fit an exact sequence of $A_\frp[\cG_{K_{v_i}}]$-modules
\[
0\to T_\frp\ul{G}' \to T_\frp\ul{G} \to \Lambda_{v_i}\otimes_A A_\frp\to 0.
\]
Since the $\cG_{K_{v_i}}$-action on $\Lambda_{v_i}$ factors through a finite quotient of $\cG_{K_{v_i}}$,  
after replacing $K$ by a finite separable extension if necessarily, we may assume that 
$\cI_{K_{v_i}}$ acts trivially on $\Lambda_{v_i}$. 
This shows the lemma.
\end{proof}

\begin{lemma}\label{lemma.ext}
Let $\wt K_\ur$ be the largest subextension of $\wt K/K$ such that all $v\in \Sigma_K$ are unramified.
It follows that 
\begin{itemize}
\item[(1)] $\wt K_\ur$ is a finite extension of $K$, and 
\item[(2)] for each $\frp\in \Sigma_F$, the composite $\wt K_\ur K(\frp^\infty)$ is the maximal extension of $K$ in $\wt K$ which is unramified at all $v\in \Sigma_K$ with $v\nmid \frp$. 
\end{itemize}
\end{lemma}
\begin{proof}
For each $\frl\in \Sigma_F$, we
put $\cG_\frl:=\Gal(F(\frl^\infty)/F)$.
Fix $\frp\in \Sigma_F$.
Then we have $\Gal(\wt F/F)\cong \cG_\frp\times \prod_{\frl\neq \frp}\cG_\frl$.
Since the restriction map $\Gal(\wt K/K) \to \Gal(\wt F/F)$ is injective, 
we may view $\Gal(\wt K/K)$ as a subgroup of $\cG_\frp\times \prod_{\frl\neq \frp}\cG_\frl$.
Let $\cI$ (resp.\ $\cJ$) be the subgroup of $\Gal(\wt K/K)$ generated by all inertia subgroups of $\Gal(\wt K/K)$ at $v\in \Sigma_K$ dividing $\frp$ (resp.\ not dividing $\frp$).
Then $\cI$ and $\cJ$ have finite index in $\cG_\frp$ and $\prod_{\frl\neq \frp}\cG_\frl$, respectively.
Hence the direct product $\cI\times \cJ$, which is  the subgroup generated by all inertia subgroups of $\Gal(\wt K/K)$ at finite places, has finite index in $\Gal(\wt K/K)$.
By $\cI\times \cJ=\Gal(\wt K/\wt K_\ur)$, we have (1).
Since the fixed field $\wt K^{\cJ}$ is the maximal subextension of $\wt K/K$ unramified at  all $v\in \Sigma_K$ with $v\nmid \frp$, we have 
\begin{align*}
\Gal(\wt K/\wt K^{\cJ})
=\cJ
&=(\cI\times\cJ)\cap \prod_{\frl\neq \frp}\cG_\frl\\
&=\Gal(\wt K/\wt K_\ur)\cap \Gal(\wt K/K(\frp^\infty))\\
&=\Gal(\wt K/\wt K_\ur K(\frp^\infty)).
\end{align*}
This implies (2).
\end{proof}

Let $K_0$ be the separable closure of $F$ in $K$.
We remark that $K/K_0$ is purely inseparable and hence all places are totally ramified.
For each place $v$ of $K$, we denote by $e_v$ the ramification index of $v$ in $K/F$. If $v$ is unramified in $K_0/F$, we obtain 
\[
e_v=[K:F]_{\mathrm{insep}}=[K:K_0].
\]
By Lemmas \ref{lemma.unip} and \ref{lemma.ext}, in the proof of Theorems \ref{theorem.R1} and \ref{theorem.R2},  we may assume that
\begin{itemize}
\item[(A1)] for each $v\in \Sigma_K$ and $\frp\in \Sigma_F$ with $v\nmid \frp$,  the $\cI_{K_v}$-action on $T_\frp\ul G$ is unipotent, and  
\item[(A2)] for each $\frp\in \Sigma_F$, $K(\frp^\infty)$ is the largest extension of $K$ in $\wt K$ unramified 
at all $v\in \Sigma_K$ with $v\nmid \frp$.
\end{itemize}

\begin{proof}[Proof of Theorem \ref{theorem.R1}]
It suffices to show that $\ul G[\frp](\wt K)=0$ for almost all $\frp$.
Now we find $[K:K_0]=p^m$ for some $m\geq 0$.
Let $\Sigma$ be the subset of $\Sigma_F$ consisting of all $\frp$ such that
\begin{itemize}
\item $\frp$ is unramified in $K_0/F$,
\item $q^{\deg \frp}-1>p^m$, and
\item $\ul G$ has good reduction at a finite place of $K$ dividing $\frp$.
\end{itemize}
Since $\Sigma_F\setminus \Sigma$ is finite, we shall consider only $\frp$ in $\Sigma$.

For  a prime $\frp\in \Sigma$, we assume that $\ul G[\frp](\wt K)\neq 0$, so that $(T_\frp\ul{G}/\frp T_\frp\ul{G})^{\cG_{\wt K}}\neq 0$.
Let $W\subset \ul G[\frp](\wt K)$ be  a non-zero irreducible $\bF_\frp$-subrepresentation of $\Gal(\wt K/K)$.
It follows by (A1) that $W$ is unramified at any finite place of $K$ not dividing $\frp$.
Hence the $\Gal(\wt K/K)$-action on $W$ factors through $\Gal(K(\frp)/K)$ by (A2).
Schur's lemma implies that $\bF:=\End_{\bF_\frp[\Gal(K(\frp)/K)]}(W)$ is a finite field containing $\bF_\frp$.
Since $\Gal(K(\frp)/K)$ is abelian, its action on $W$ is given by a character
\[
\bar \rho \colon \Gal(K(\frp)/K)\lra \bF^\times.
\]
Here $\Gal(K(\frp)/K)$ has order dividing $\#\bF_\frp^\times=q^{\deg \frp}_{}-1$ and hence 
$\bar \rho$ has values in $\bF_\frp^\times$.
This implies that the $\bF_\frp$-dimension of $W$ is one because $W$ is irreducible.
Thus there is an integer $0\leq n < q^{\deg\frp}-1$ such that 
\[
\bar \rho=\bar \chi_\frp^n,
\] where $\bar\chi_\frp\colon \Gal(K(\frp)/K)\to \bF_\frp^\times$ is the  character coming from $\ul C[\frp](K^\sep)$.
We claim that $0\leq n\leq p^m$.
To see this, we take  $v\in \Sigma_K$ dividing $\frp$ such that $\ul G$ has good reduction over $K_v$.
Note that $e_v=p^m$.
Denote by $R$ the valuation ring of $K_v$.
Let $\ul{\wh M}=(\wh M, \tau_{\wh M})$ be a local shtuka over $R$ corresponding to the $A$-motive $\ul M(\ul G)_{K_v}$.
By construction, $\ul{\wh M}$ is effective \textit{of height $\leq 1$} 
meaning that $(z_\frp-\zeta_\frp)\wh M\subset \tau_{\wh M}(\wh\sig\ast \wh M) \subset \wh M$.
Note that $W|_{\cG_{K_v}}$ can be viewed as an $\bF_\frp$-subrepresentation of $\Hom_{A_\frp}(\wc T_\frp\ul{\wh M}, A_\frp)\ot_{A_\frp}\bF_\frp$. 
It follows by \cite[Proposition 4.23]{Oku24} that $\bar\rho|_{\cI_{K_v}}=\theta_1^j$ with $0\leq j \leq e_v=p^m$, where $\theta_1\colon \cI_{K_v} \to \bF_\frp^\times$ is the \textit{fundamental character of level one} (see \cite{Ser72} and \cite[\S\S 4.2]{Oku24}).
Since $\bar\chi_\frp|_{\cI_{K_v}}=\theta_1$,  we have $n\equiv j\pmod {q^{\deg\frp}-1}.$
By $e_v=p^m<q^{\deg \frp}-1$, we get $n=j$ as claimed.

Consequently, if $\ul G[\frp](\wt K)\neq 0$ for some $\frp\in \Sigma$, then $\ul G[\frp](\wt K)$ has an $\bF_\frp$-subspace $W$ satisfying either 
\begin{itemize}
\item[(i)] the $\Gal(\wt K/K)$-action on $W$ is trivial, or
\item[(ii)] $\Gal(\wt K/K)$ acts on $W$ via $\bar{\chi}_\frp^n$ for some $1\leq n \leq p^m$.
\end{itemize}
There are only finitely many $\frp \in \Sigma$ such that (i) occurs, because if not, then $\ul G(K)_{\tor}$ is infinite; but this contradicts \cite[Proposition 1]{Poo97}.
We assume that there are infinitely many $\frp\in \Sigma$ for which (ii) occurs.
Then there are an infinite subset $\Sigma' \subset \Sigma$ and an integer $n'$ with $1\leq n' \leq  p^m$ such that, for all $\frp\in \Sigma'$, $\ul G[\frp](\wt K)$ has an $\bF_\frp$-subrepresentation whose Galois action is described by 
$\bar{\chi}^{n'}_\frp$ .
We can take a finite place $v$ of $K$ such that it does not divide all $\frp\in \Sigma'$ and  $\ul G$ has good reduction at $v$.
Put 
\[
\ul M:=\ul M(\ul G)^\vee\otimes \ul{M}(\ul C)^{\ot n'},
\]
 which is an $A$-motive over $K$ of rank equal to $\rk \ul{G}$.
Since $\wc T_\frp \ul{M}(\ul G)^\vee\otimes_{A_\frp}\bF_\frp$ contains a submodule isomorphic to $\ul G[\frp](\wt K)$, 
we see that $\wc T_\frp\ul M/\frp \wc T_\frp\ul M$ contains a non-zero submodule with trivial $\cG_{K}$-action.
In particular, we have
$(\wc T_\frp\ul M/\frp \wc T_\frp\ul M)^{\cG_{K_v}}\neq 0$ for all $\frp\in\Sigma'$.
By construction, $\ul M_{K_v}$ has good reduction with non-zero weights 
\[
n'-\frac{1}{r},\ldots, n'-\frac{1}{r}\es\es\es (\text{multiplicity}\ r),
\]
 where $r=\rk \ul{G}\geq 2$.
However, this must not happen by Proposition \ref{proposition.frl} because $(\wc T_\frp\ul M/\frp \wc T_\frp\ul M)^{\cG_{K_v}}\neq 0$ implies $(\wc V_\frp\ul M/\wc T_\frp\ul M)^{\cG_{K_v}}\neq 0$.
Therefore (i) and (ii) occur for only finitely many $\frp$.
The theorem is proved.
\end{proof}

\begin{proof}[Proof of Theorem \ref{theorem.R2}] 
Suppose that $\ul G[\frp^\infty](\wt K)_\tor$ is infinite for some $\frp\in \Sigma_F$.
Then  
$(V_\frp\ul {G})^{\cG_{\wt K}}\neq 0$.
Take a non-zero irreducible $Q_\frp$-subrepresentation $V\subset (V_\frp\ul {G})^{\cG_{\wt K}}$ of $\cG_{K}$. 
Since $V$ is unramified at all finite places of $K$ not dividing $\frp$ by (A1), the  $\cG_{K}$-action factors through 
$\Gal(K(\frp^\infty)/K)$ by (A2).
Take a place $v\in \Sigma_K$ dividing $\frp$.
We put $L:=K_v$.
Then we have $K(\frp^\infty)\subset L_\infty:=L\bF_\frp\dpl \zeta_\frp\dpr_\infty$ and hence  the $\cG_L$-action on $V$ factors through $\Gal(L_\infty/L)$.
It follows by (A1) that 
the $A$-motive $\ul M(\ul G)$ has \textit{strongly semi-stable reduction} over $L$; see \cite[Definition 4.6]{Gar03} and \cite[Definition 5.8]{Oku24}.
By \cite[Proposition 5.15]{Oku24}, the semi-simplification $(V_\frp\ul G|_{\cG_L})^\ss$
of $V_\frp\ul G|_{\cG_L}$ is a direct sum of $z_\frp$-adic crystalline representations of $\cG_L$. 
Since $V$ is irreducible as a $Q_\frp$-representation of the abelian group $\Gal(K(\frp^\infty)/K)$, we see that $V|_{\cG_L}$ is also semi-simple.
Hence $V|_{\cG_L}$ is a direct summand of $(V_\frp\ul G|_{\cG_L})^\ss$; in particular, it is also crystalline.
Since $V|_{\cG_L}$ is fixed by $\Gal(L_\infty/L)$ and of dimension $\leq 2$, Lemmas \ref{lem.char} and \ref{lemma.key} imply that 
$Q_\frp(n)\subset V|_{\cG_K}$ for some integer $n$, after replacing $L$ by a finite separable extension.
In particular, we have $(V_\frp\ul{G}|_{\cG_L})^{\cG_{L_\infty}}\neq 0$.
This means that $\ul{G}[\frp^\infty](L_\infty)$ is infinite, which contradicts  Corollary \ref{corollary.Amod}
\end{proof}

%
%

\bibliographystyle{amsalpha} 
\bibliography{Okumura}

\vspace{50pt}

Department of Architecture, Faculty of Science and Engineering,
Toyo University

2100, Kujirai, Kawagoe, 
Saitama 350-8585, Japan

{\it  E-mail address} : \email{\tt okumura165@toyo.jp}

{\it URL} : {\tt \url{https://sites.google.com/view/y-okumura/index-e?authuser=0}}
\end{document}